\documentclass[12pt]{article}

\usepackage{latexsym,amsfonts,amsmath,amssymb,wasysym,stmaryrd,enumerate,epsfig}

\usepackage{mathrsfs}

\usepackage{theorem}

\usepackage{tikz}
\usetikzlibrary{matrix,arrows,shapes,calc}

\numberwithin{equation}{section}

\newtheorem{definition}{Definition}[section]

\newtheorem{proposition}[definition]{Proposition}
\newtheorem{theorem}[definition]{Theorem}

\newtheorem{lemma}[definition]{Lemma}

\newenvironment{remark}[1][Remark:]{\begin{trivlist}
\item[\hskip \labelsep \emph{#1}]}{\end{trivlist}}

\newenvironment{example}[1][Example:]{\begin{trivlist}
\item[\hskip \labelsep \emph{#1}]}{\end{trivlist}}

\newcommand{\be}{\begin{equation}}
\newcommand{\ee}{\end{equation}}
\newcommand{\beu}{\begin{equation*}}
\newcommand{\eeu}{\end{equation*}}
\newcommand{\bea}{\begin{eqnarray}}
\newcommand{\eea}{\end{eqnarray}}
\newcommand{\beaa}{\begin{eqnarray*}}
\newcommand{\eeaa}{\end{eqnarray*}}
\newcommand{\bmx}{\begin{pmatrix}}
\newcommand{\emx}{\end{pmatrix}}

\newcommand{\g}{{\mathfrak g}}

\newcommand{\finproof}{{\hfill \rule{5pt}{5pt}}}

\newcommand{\mf}{\mathfrak}

\newcommand{\alf}{{\textstyle{\frac{1}{2}}}}
\newcommand{\half}{\frac{1}{2}}

\newcommand{\nn}{\nonumber}
\newcommand{\sign}{{\rm sign}}

\newcommand{\8}{{\infty}}

\newcommand{\eps}{\epsilon}

\newcommand{\rank}{{\rm rank}}

\newcommand{\Z}{{\mathbb Z}}\newcommand{\CC}{{\mathbb C}}

\newcommand{\ket}[1]{{\,\left|#1\right>}\,}

\newcommand{\id}{{\mathrm{id}}}
\newcommand{\bl}{{\bullet}}
\newcommand{\wh}{{\circ}}

\newcommand{\uqslt}{{\mathrm{U}_q({\mathfrak{sl}}_2})}
\newcommand{\uqslth}{{\mathrm{U}_q(\widehat{\mathfrak{sl}}_2})}
\newcommand{\uqgh}{{\mathrm{U}_q(\widehat\g)}}

\newcommand{\uqlslt}{{\mathrm{U}_q(\mathrm L{\mathfrak{sl}}_2})}

\newcommand{\lambdaof}[1]{\lambda^{#1}}

\newcommand{\comp}[1]{\overline{#1}}

\newcommand{\sgn}{\varepsilon}

\newcommand{\goi}[2]{=}

\newcommand{\on}{}

\newcommand{\groth}[1]{{\mathrm{\bf Mod}(#1)}}
\newcommand{\Cx}{\mathbb C^*}

\newcommand{\grass}[2]{\mbox{\bf Gr}(#1, #2)}
\newcommand{\grassm}[2]{\mbox{\bf Gr}^\bullet(#1, #2)}

\newcommand{\btp}{\begin{tikzpicture}[baseline=0pt,scale=0.9,line width=0.25pt]}
\newcommand{\etp}{\end{tikzpicture}}
\newcommand{\bsc}[2]{\begin{tikzpicture}[baseline=0pt,scale=0.5,line width=0.25pt]
\draw[gray] (0,0) grid (#1,#2);}
\newcommand{\esc}{\end{tikzpicture}}
\newcommand{\bscl}[2]{\begin{tikzpicture}[baseline=0pt,scale=1,line width=0.25pt]
\draw[gray] (0,0) grid (#1,#2);}
\newcommand{\escl}{\end{tikzpicture}}

\newcommand{\range}[2]{\left[ \! \left[ #1,#2 \right] \! \right]}
\newcommand{\ijr}[1]{1\leq i,j \leq #1}
\newcommand{\ir}[1]{1\leq i \leq #1}

\newcommand{\concat}{\#}

\newcommand{\phiin} {\phi^\pm_{\bl}}
\newcommand{\phiout}{\phi^\pm_{\wh}}
\newcommand{\phimid}{\phi^\pm_{H}}

\newcommand{\VV}[1]{\ket{V_{#1}}}
\newcommand{\ZZ}[1]{\ket{Z_{#1}}}
\newcommand{\V}{\mathsf V}

\newcommand{\ais}{{(a_1, \dots, a_n)}}

\newcommand{\down}{\downarrow}
\newcommand{\binomt}[2]{\binom{#1}{#2}_{\! t}}

\newcommand{\Vmn}{\V_{m,n}}
\newcommand{\Vmmmn}{\V_{(m),(n)}}
\newcommand{\psimnur}{{\psi^\pm_{m,n;\pm r}}}
\newcommand{\psimnu}{{\psi^\pm_{m,n}(u)}}
\newcommand{\psimmmnu}{{\psi^\pm_{(m),(n)}(u)}}
\newcommand{\ie}{\textit{i.e.}\,\,}
\newcommand{\eg}{\textit{e.g.}\,\,}

\newcommand{\sqr}[1]{\draw[thick] #1 rectangle ++(1,1); \draw[thick]   #1 ++(.25,.5) -- ++(.1,.1)-- ++(-.1,-.1) -- ++(.1,-.1) -- ++(-.1,.1) -- ++ (.5,0) ;}
\newcommand{\sqd}[1]{\draw[thick, fill=gray!20] #1 rectangle ++(1,1); \draw[thick]  #1 ++(.5,.75) -- ++(.1,-.1) -- ++(-.1,.1) -- ++(-.1,-.1) -- ++(.1,.1) -- ++ (0,-.5) ;}

\advance\textheight by \topskip
\begin{document}

\baselineskip 16.5pt
\parindent 18pt
\parskip 8pt

\begin{titlepage}
\begin{flushright}
{\bf \today} \\
YITP-10-30\\
LPT-10-26
\end{flushright}
\begin{centering}
\vspace{.2in}

{\Large {\bf 
On $q,t$-characters and the $l$-weight Jordan \\ filtration of standard $\uqslth$-modules
}}
\vspace{.3in}

{\large C. A. S. Young ${}^{1}$ and R. Zegers ${}^2$}\\    
\vspace{.2 in}    
${}^{1}$ {\emph{Yukawa Institute for Theoretical Physics\\  
Kyoto University, Kyoto\\
 606-8502, Japan}}\\    
\vspace{.2 in}    
${}^{2}$ {\emph{Laboratoire de Physique Th\'eorique\\ Universit\'e Paris-Sud 11/CNRS\\ 91405 Orsay Cedex, France}}    
    
%
\footnotetext[1]{{\tt charlesyoung@cantab.net,}\quad ${}^{2}${\tt robin.zegers@th.u-psud.fr}
}
\vspace{.5in}

{\bf Abstract}

\vspace{.1in}

\end{centering}

{The Cartan subalgebra of the quantum 
loop algebra $\uqlslt$ is generated by a family of mutually commuting operators, responsible for the $l$-weight decomposition of finite dimensional $\uqlslt$-modules. 
The natural Jordan filtration induced by these operators is generically non-trivial on $l$-weight spaces of dimension greater than one. 
We derive, for every standard module of $\uqlslt$, the dimensions of the Jordan grades and prove that they can be directly read off from the $t$-dependence of the $q,t$-characters introduced by Nakajima. 
To do so we construct explicit bases for the standard modules of $\uqlslt$ with respect to which the Cartan generators are upper-triangular. The basis vectors of each $l$-weight space are labelled by the elements of a ranked poset from the family $L(m,n)$.
}

\end{titlepage}

\begin{flushright}
\break

\end{flushright}
\vspace{1cm}
\section{Introduction}
Let $q \in \Cx$, not a root of unity.
Let $\g$ be a simple finite dimensional Lie algebra and $\widehat{\g}$ be the corresponding (untwisted) affine Ka\v c-Moody algebra. The quantum affine algebra $\uqgh$ can equivalently be seen as the Drinfel'd-Jimbo, or \emph{standard}, quantum universal enveloping algebra of $\widehat{g}$ \cite{Drinfeld1, Jimbo}; or as the \emph{quantum affinization} of  the standard quantum universal enveloping algebra $\mathrm{U}_q (\g)$ \cite{Drinfeld2}. The latter point of view, sometimes referred to as Drinfel'd's current -- or second, or new -- presentation, is particularly well adapted to the study of finite dimensional representations. In particular, it provides a triangular decomposition $(\mathrm{U}^-_q(\widehat \g), \mathrm{U}^0_q(\widehat\g), \mathrm{U}^+_q(\widehat\g))$ with respect to which all finite dimensional $\uqgh$-modules turn out to be highest weight in a generalized sense \cite{CPsl2,CPbook}. The corresponding generalization of the notion of weight, often referred to a
 s \emph{$l$-weights}, consists in $\rank(\g)$-tuples of rational functions of a formal variable $u$. The finite dimensional simple $\uqgh$-modules 
were classified by Chari and Pressley in terms of their highest $l$-weights, the so-called Drinfel'd polynomials. Their complete $l$-weight structure was then shown to be encoded in a \emph{$q$-character}, \cite{FR} -- see also \cite{Knight} --, generalizing the standard theory of characters to the quantum affine context. 
There exists an algorithm, due to Frenkel and Mukhin \cite{FM}, to compute the $q$-character of any module with a unique dominant $l$-weight. This too has a direct analog in classical representation theory; however, such modules are much more ubiquitous in the quantum affine case -- essentially because the loop variable $u$ lifts many degeneracies --, making the algorithm correspondingly more powerful.

The $q$-character $\chi_q(V)$ of a given $\uqgh$-module $V$ is a polynomial in an infinite set of formal variables $(Y_{i,a}^{\pm 1})_{1\leq i\leq \rank(\g), a \in \Cx}$ such that with each $l$-weight space $V_\gamma$ of $V$ is associated a unique monomial $m_\gamma$ in the expression of $\chi_q(V)$, whose coefficient is equal to the dimension of $V_\gamma$. When the latter is one-dimensional it constitutes a trivial $\mathrm{U}^0_q(\widehat\g)$-module. Thus, when all the $l$-weight spaces of a $\uqgh$-module $V$ are one dimensional -- we shall say that $V$ is \emph{thin}, see definition \ref{Def:Thin} --, its structure as a $\mathrm{U}^0_q(\widehat\g)$-module can be directly read off from its $q$-character. In contrast, when an $l$-weight space $V_\gamma$ has dimension strictly greater than one, or, equivalently, when the associated monomial $m_\gamma$ occurs in the expression of a $q$-character with a coefficient greater than one -- we shall say that the module to which it 
 pertains is \emph{thick} --, its structure as a $\mathrm{U}^0_q(\widehat\g)$-module is expected to become non-trivial. This feature, which has no direct analog in the representation theory of quantum groups of finite type,
is related to the formation of non-trivial Jordan blocks for the action of the $l$-weight generators $(\phi^\pm_i(u))_{1\leq i\leq \rank(\g)}$ of $\mathrm{U}^0_q(\widehat\g)$ on $V_\gamma$. In that case, restricting to $\uqslth$ for simplicity -- as we shall throughout this paper --, $V_\gamma$ is expected to admit a non-trivial Jordan filtration.
The simplest non-trivial example is the $\uqslth$-module $V_a \otimes V_a$, where $V_a$ is the fundamental $\uqslth$-module with spectral parameter $a \in \Cx$. It was discussed at the end of \cite{HernandezFusionI}. Except for this example, the actual structure of these Jordan filtrations is unknown. This clearly constitutes an important gap in the present understanding of the structure of finite-dimensional representations of quantum affine algebras, especially since many representations of interest are thick. For example, minimal affinizations are thin in types $A$ and $B$ but not in other types in general \cite{Hminaffs}.

The first purpose of this paper is thus to determine the Jordan filtration of the $l$-weight spaces for \emph{standard} modules of $\uqslth$. We do so in Theorem \ref{THX}. Standard modules were originally introduced in the context of $K$-theoretical realizations of quantum affine algebras and of their finite dimensional representations \cite{GV, Vass, Nakajima1} -- see section \ref{Sec:Def} for an algebraic characterization due to \cite{VV}, and section \ref{Sec:qtchar} for a definition in terms of the Borel-Moore homology of Grassmannians. They constitute a basis of the Grothendieck ring $\groth\uqslth$ of the category of finite dimensional $\uqslth$-modules. Moreover, there is a certain sense in which one can continuously interpolate between every thick standard module and the generic, thin, one. This property will allow us to construct explicitly thick standard modules as appropriate limits of thin standard modules and subsequently to determine the Jordan filtrations of t
 heir $l$-weight spaces. 

As a corollary of Theorem \ref{THX}, we establish a correspondence between the dimensions of the associated Jordan grades and the \emph{$q,t$-characters} introduced by Nakajima. This proves in type $A_1$ a slightly modifed version of a conjecture in \cite{Nakajima2}. As we recall in section 5, $q,t$-character  were first defined geometrically for standard modules over quantum affine algebras of simply laced type, as generating functions of Poincar\'e polynomials for graded quiver varieties of corresponding type \cite{Nakajima2}. This geometrical definition was subsequently axiomatized in \cite{Nakajima3} and extended to arbitrary Lie algebras in \cite{Hernandezqt}. $q,t$-characters constitute a generalization of $q$-characters in the sense that they evaluate to $q$-characters in the limit $t=1$. The formal parameter $t$ thus appears as a way to lift the degeneracies of the monomials in the $q$-characters of thick $\uqgh$-modules and it is natural to expect that coefficients o
 f 
 powers of $t$ in $q,t$-characters encode the Jordan filtrations of the $l$-weight spaces of $\uqgh$-modules in some way. The only discrepancy in our result as compared to the conjecture \cite{Nakajima2} is the necessity to reorder these coefficients into weakly decreasing order. Within any given $l$-weight space $V_\gamma$, the dimensions of the associated Jordan grades are certainly weakly increasing with $k$: intuitively speaking, in a Jordan basis each Jordan chain ``begins'' at some grade and contributes $+1$ to the dimension of all subsequent grades until one reaches the eigenspace. 

It is worth noting that, as it stands, theorem \ref{THX} does not tell us anything about those thick $\uqslth$-modules that occur as non-trivial \emph{quotients} of thick standard modules (including those thick simple modules that are isomorphic to tensor products of non-trivial Kirillov-Reshetikhin modules). 
Let us finally mention that, in \cite{VVPervSheav}, $q,t$-characters were also proven to encode a gradation of the $l$-weight spaces compatible with the action of a family of generators in $\mathrm{U}^-_q(\widehat\g)$ -- denoted $\phi_{ir}^{(t)}$ there -- which should not be mistaken for the one based on the $l$-weight operators that we consider here.

The structure of this paper is as follows. 
In section \ref{Sec:Def}, we recall the necessary facts about $\uqslth$ and its finite dimensional representations, including the algebraic definition of standard modules. At the end of section \ref{Sec:Def}, we choose a basis of $l$-weights of a general thin standard module and give explicitly the action of the generators of $\uqslth$ in this basis. 
In sections \ref{Sec:ZA} and \ref{Sec:multiplepoints} we find explicit bases for all thick standard modules, by regarding them as limiting cases of the thin standard modules in which the spectral parameters coincide. In particular, we compute the lengths and multiplicities of the Jordan chains of $\phi^\pm(u)$. 
Finally, in section \ref{Sec:qtchar} we recall the geometrical construction of standard modules and the definition of $q,t$-characters. This allows us to show that the latter do indeed encode the Jordan block structure of $\phi^\pm(u)$. We conclude by commenting on the geometrical interpretation of our results and what one could expect in higher rank cases.

\section{The quantum affine algebra $\uqslth$ and its finite dimensional representations}
\label{Sec:Def}
The quantum affine algebra $\uqslth$ can be seen as the quantum affinization of the finite dimensional quantum algebra $\uqslt$. From this point of view, $\uqslth$ is 
an associative algebra over $\mathbb C$ generated by
\be k^{\pm 1}, \quad\quad  (h_{n})_{n \in \mathbb Z^*}\, , \quad\quad (x^\pm_{n})_{n \in \mathbb Z}, \quad\quad c^{\pm 1/2} \, .\label{gens}\ee
In this paper, since we are only concerned with finite dimensional representations of $\uqslth$ and, since it suffices \cite{CPbook} to see these as representations of the quantum loop algebra $\uqlslt$, we shall, as usual,
ignore the 
central elements $c^{\pm 1/2}$. Arranging the remaining generators into the following formal power series
\be x^\pm(u) := \sum_{n \in \mathbb Z} x^{\pm}_{n} u^{-n} \label{xiu} \ee    
\be \phi^\pm (u) = \sum_{n=0}^\infty  \phi_{\pm n}^\pm u^{\pm n} :=     
k^{\pm 1} \exp \left(\pm (q - q^{-1}) \sum_{m=1}^\infty h_{\pm m} u^{\pm m} \right ) \, , \label{phiu} \ee    
the defining relations of $\uqlslt$ read
\be    
\left [\phi^\pm(u) , \phi^\pm (v) \right ] = \left [ \phi^\pm (u) , \phi^\mp(v) \right ] = 0 \label{phiphi}\ee    
\be    
(q^{-1}- quv) \phi^\pm(u)\,  x^+(v) = (q- q^{-1} uv) \, x^+(v) \, \phi^\pm(u) \label{phix+}\ee    
\be    
(q- q^{-1}uv)\phi^\pm(u) \, x^-(v) = (q^{-1}- q uv) \, x^-(v) \, \phi^\pm(u) \label{phix-}    
\ee    
\be \left [x^+(u), x^-(v) \right ] = \frac{1}{q-q^{-1}} \left (\delta(v/u) \phi^+(1/v) - \delta(u/v) \phi^-(1/u) \right ) \label{x+x-}\ee    
\be \left ( u-q^{\pm 2}v \right )x^\pm(u) \, x^\pm(v) =     
\left (q^{\pm 2}u-v \right)\, x^\pm(v) \, x^\pm(u) \, ,\label{x+x+1}\ee    
where we have set
\be \delta(u) := \sum_{n \in \mathbb Z} u^n .\ee
For the standard coproduct in the current presentation above see \cite{Thoren2000iro}. There also exists a so-called \emph{current coproduct}:
\be \Delta (\phi^\pm(u)) = \phi^\pm(u) \otimes \phi^\pm(u) \label{deltaphi}\ee
\be \Delta (x^+(u)) = 1 \otimes x^+(u) +  x^+(u) \otimes \phi^-(1/u) \label{deltax+}\ee
\be \Delta (x^-(u)) =x^- (u) \otimes 1 +  \phi^+(1/u) \otimes   x^-(u) \label{deltax-}\,;\ee
although one should keep in mind that this is, strictly speaking, ill-defined because of the infinite sums involved on the r.h.s. of (\ref{deltax+}) and (\ref{deltax-}) -- see \eg \cite{Grosse,HernandezFusionII} for a rigorous treatment. 

For any $\uqslth$-module $V$ and any formal series
\be \gamma^\pm(u) := \sum_{m \in \mathbb N}  \gamma_{\pm m}^\pm u^{\pm m} \in \mathbb C[[u^{\pm 1}]] \ee
define 
\be V_{\gamma} = \{ v \in V : \exists N \in \mathbb N\,\,\mathrm{s.t.}\,\, \forall r\in\mathbb N_0 \quad \left( \phi^\pm_{\pm r} - \gamma^\pm_{\pm r}\, \id\right)^N \on v = 0 \} \, . \label{lweightspacesdef}\ee
Whenever $\dim V_{\gamma}>0$, one says that $\gamma$ is an \emph{$l$-weight} of $V$ and that $V_{\gamma}$ is the corresponding \emph{$l$-weight space}. A representation of $\uqslth$ is of \emph{type 1} if it is the direct sum of its $l$-weight spaces (and $c^{\pm1/2}$ acts as the identity on it). 
It is known \cite{FR} that for every finite-dimensional type 1 representation of $\uqslth$, every $l$-weight is of the form
\be \gamma^\pm(u) = q^{\deg Q - \deg R} \, \frac{Q(uq^{-1}) R(uq)}{Q(uq) R(uq^{-1})} \, ,\ee
where the right hand side is to be treated as a series in positive (resp. negative) integer powers of $u$ for $\gamma^+(u)$ (resp. $\gamma^-(u)$), and $Q$ and $R$ are monic polynomials\footnote{\ie polynomials with constant term $1$},
\be Q(u) = \prod_{a\in \Cx} \left( 1- ua\right)^{q_{a}}, \quad\quad 
    R(u) = \prod_{a\in \Cx} \left( 1- ua\right)^{r_{a}}. \ee
The above property allows one to give a purely algebraic definition of $q$-characters: assigning to $\gamma$ a monomial
\be m_{\gamma} = \prod_{a\in \Cx} Y_{a}^{q_{a}-r_{a}}\ee
in formal variables $(Y_{a}^{\pm 1})_{a\in \Cx}$, the \emph{$q$-character} map $\chi_q$ \cite{FR} is the homomorphism of rings
\be \chi_q : \groth\uqslth \longrightarrow \mathbb Z\left[ Y_{a}^{\pm 1} \right]_{a \in \Cx} \ee
defined by
\be \chi_q (V) = \sum_{\gamma} \dim\left(V_{ \gamma}\right) m_{\gamma}.\ee
It is injective \cite{FR}. 

As we argued in the introduction, the $q$-character of a given $\uqslth$-module characterizes its structure as a $\mathrm{U}^0_q(\widehat{\mathfrak{sl}}_2)$-module more or less completely depending on the dimensions of its different $l$-weight spaces. The crucial distinction is given by the following
\begin{definition}
\label{Def:Thin}
A $\uqslth$-module of type 1 is \emph{thin} if and only if all its $l$-weight spaces have dimension $1$. Otherwise, it is \emph{thick}.
\end{definition}
While knowledge of the $q$-character of a thin module completely characterizes its structure as a $\mathrm{U}^0_q(\widehat{\mathfrak{sl}}_2)$-module, it only gives partial information as to the $\mathrm{U}^0_q(\widehat{\mathfrak{sl}}_2)$-module structure of thick modules. 
\begin{definition}
The fundamental $\uqslth$-module at spectral parameter $a \in \Cx$, $V_a$, is the simple $\uqslth$-module with Drinfel'd polynomial $P(u)= (1-au)$.
\end{definition}
It may be shown that
\be\label{qcharfund} \chi_q(V_a) = Y_a + Y_{aq^2}^{-1}  \ee 
Thus, $V_a$ is thin. Furthermore, it is simple and the simplest example of a \emph{standard module}. 
A theorem due to Varagnolo and Vasserot \cite{VV} provides (for $q \in \Cx$ not a root of unity, as here) a purely algebraic identification of standard modules. In the case of $\uqslth$, it allows us to make the following
\begin{definition}\label{Mdef}
Let $P\in\mathbb C[u]$ be a monic polynomial. Let $(a^{-1}_i)_{1\leq i\leq \deg P}$ be the roots of $P$ arranged in any order such that $\ell<k$ $\implies$ $a_\ell/a_k \neq q^2$. 
Then the \emph{standard module} $M(P)$ is defined by
\be \label{MPdef} M(P) \cong V_{a_1} \otimes V_{a_2} \otimes \dots \otimes V_{a_{\deg P}},\ee
(where the right-hand side is an ordered tensor product). 
\end{definition}
(The non-trivial result is 
that the ordering imposed on the $a_{i}$ ensures that this definition makes sense not only at the level of equivalence classes in $\groth\uqslth$ but even at the level of isomorphism classes in the category of finite dimensional $\uqslth$-modules itself. Note that, in our conventions for the coproduct, $V_{aq^{-1}}\otimes V_{aq}$ admits a trivial submodule and is highest $l$-weight, whereas $V_{aq}\otimes V_{aq^{-1}}$ admits a trivial quotient module and is not, c.f. \cite{CPsl2}.)

The following proposition is adapted from \cite{Nakajima1}, proposition 13.3.1.
\begin{proposition}
\label{Prop:Stdhighest}
Let $P\in \mathbb C[u]$ be a monic polynomial. The standard module $M(P)$ is a highest $l$-weight module with highest $l$-weight
\be \gamma_P^\pm(u) = q^{\deg P}\frac{P(uq^{-1})}{P(uq)} \, ,\ee
namely, there exists $v \in M(P)$ such that the following hold:
\begin{enumerate}[i)]
 \item $x^+(u) \on v = 0$;
\item $\left ( \phi^\pm(u) - \gamma_P^\pm(u) \id \right ) \on v = 0$;
\item $M(P)= \mathrm{U}^-_q({\widehat{\mathfrak{sl}}_2}) \on v$
\end{enumerate}
\end{proposition}

\noindent As already mentioned, standard modules need not be simple in general. However, we have
\begin{proposition}[\cite{CPsl2}]
\label{Prop:StdSimple}
Let $P(u) = \prod_{i=1}^{\deg P} (1-a_i u)$, with $a_i\in \Cx$ for $1\leq i\leq\deg P$. The standard module $M(P)$ is simple --- and therefore isomorphic to $L(P)$ --- if and only if $i\neq j$ $\implies$ $a_i/a_j\notin \{q^2,q^{-2}\}$.
\end{proposition}

\noindent Thus, for \emph{generic} $P$, $L(P) \cong M(P)$. More generally, for every Drinfel'd polynomial $P$, the simple module $L(P)$ is the unique simple quotient of $M(P)$.
So, when $P$ is such that $M(P)$ is not simple, 
$L(P)$ becomes a \emph{non-trivial} quotient of $M(P)$. This produces a discontinuity in the dimension of $L(P)$ while smoothly varying the roots of $P$.
\begin{lemma}[\cite{CPweyl}]
\label{lemma:stdunique}
Let $P \in \mathbb C[u]$ be a monic polynomial. $M(P)$ is the unique -- up to isomorphism -- highest $l$-weight $\uqslth$-module with $q$-character equal to $\chi_q(M(P))$
\end{lemma}

\begin{proposition}
\label{Prop:Thin}
Let $P(u) = \prod_{i=1}^{\deg P} (1-a_i u)$, with $a_i\in \Cx$ for $1\leq i\leq\deg P$.
The standard module $M(P)$ has $q$-character
\be \chi_q(M(P)) = \prod_{i = 1}^{\deg P} \left ( Y_{a_i} + Y_{a_iq^2}^{-1}\right )\ee
It is thin if and only if $i\neq j$ $\implies$ $a_i\neq a_j$.
\end{proposition}

\begin{proof}
Immediate from the definition (\ref{MPdef}) of $M(P)$ together with the fact that $\chi_q$ is a homomorphism of rings. 
\finproof
\end{proof}

\noindent Intuitively, thickness arises as a consequence of having at least two fundamental factors with equal spectral parameters. Consider, for example, $V_{aq^{-1}}^{\otimes n}$, for any $n>1$. It is thick and, indeed,
\be\label{chiqVan} \chi_q (V_{aq^{-1}}^{\otimes n}) = \sum_{k=0}^n {n \choose k} Y_{aq^{-1}}^{n-k} Y_{aq}^{-k} \, ,\ee
illustrating the pattern in which multiplicities occur in $q$-characters of $\uqslth$-modules. 

In the next sections, in order to understand the Jordan structure of thick standard modules, we shall continuously interpolate between a \emph{generic} thin standard module and every thick standard module. To do so, we introduce the following
\begin{definition}
For all $n \in \mathbb N$ and all $\ais\in (\Cx)^n$ we write
\be \V_\ais = M\left(\prod_{i=1}^n \left( 1 - a_i  q^{-1} u\right)\right).\nn\ee
\end{definition}
$\V_\ais$ varies continuously with $\ais$ in the sense discussed above and
the thick standard modules with Drinfel'd polynomial of order $n$ are simply those $\V_\ais$ for which the elements of $\ais$ fail to be pairwise distinct; we shall say that they are \emph{coincident}. In sections \ref{Sec:ZA} and \ref{Sec:multiplepoints} below, thick standard modules are therefore studied as particular \emph{coincident} limits of some thin standard module. In practice, we construct an explicit basis of $\V_\ais$, for all $\ais \in (\Cx)^n$, from the one whose existence is guaranteed by the following
\begin{proposition}
\label{Prop:VBasis}
For all $\ais \in (\Cx)^n$ whose elements are pairwise distinct, there exists a basis $(\VV A)_{A \subseteq \range 1 n}$ of the standard module $\V_\ais$, in which
\bea \phi^\pm(u)\on \VV A &=& \VV A \prod_{j\in A}    \frac{q^{-1} - q a_j u}{1-a_j u} 
                                 \prod_{j\notin A} \frac{q - q^{-1} a_j u}{1-a_j u} \label{VAphiaction}\\
         x^+(u) \on \VV A &=& \sum_{j\in A}    \delta (a_j/u) \VV{A\setminus\{j\}} 
                                                    \prod_{k \notin A} \frac{a_jq-a_kq^{-1}}{a_j-a_k} \label{VAxpaction} \\
         x^-(u) \on \VV A &=& \sum_{j\notin A} \delta (a_j/u)\VV{A\cup \{j\}} 
                                                    \prod_{k \in A}  \frac{a_k q - a_j q^{-1}}{a_k - a_j} \label{VAxmaction}.\eea
\end{proposition}
\begin{proof}
It is straightforward to check that, given (\ref{VAphiaction}-\ref{VAxmaction}), the defining relations (\ref{phiphi}-\ref{x+x+1}) are realized on the set of vectors $(\VV A)_{A \subseteq \range 1 n}$. The latter therefore spans a thin $\uqslth$-module, $V$, whose $q$-character is equal, by construction, to $\chi_q(\V_\ais)$. Furthermore, $V$ is a highest $l$-weight module: it admits $\VV{\emptyset}$ as a highest $l$-weight vector and $V= \mathrm{U}_q^-(\widehat{\mathfrak{sl}}_2) \on \VV{\emptyset}$. By lemma \ref{lemma:stdunique}, we thus have $V \cong \V_\ais$ as $\uqslth$-modules; hence the result. \finproof
\end{proof}

\section{The $n$-fold coincident limit}\label{Sec:ZA}
Proposition \ref{Prop:VBasis} provides an explicit basis of the standard module $\V_\ais$, valid whenever the elements of $\ais$ are pairwise distinct. In the remainder of the paper we consider limits in which certain elements of $\ais$ coincide.
We begin, in this section, by considering the case in which all the elements of $\ais$ tend to a common value, $a$; subsequently, in section \ref{Sec:multiplepoints}, we go to the general case.

Thus, let us consider the \emph{$n$-fold coincident limit} obtained by first setting 
\be a_i = a + z_i,\qquad z_i = \epsilon \alpha_i\qquad \forall i \in \range 1 n,\ee
for some $a\in \Cx$ and pairwise distinct $\alpha_1,\dots,\alpha_n\in\Cx$, and then sending $\eps\in\CC$ to zero. It is apparent from (\ref{VAxpaction}-\ref{VAxmaction}) that various matrix elements are singular in this limit: for example, $x^+_0 \on \VV{\{j\}} \sim \eps^{-1} \VV\emptyset$. On the other hand, the module $\V_\ais$ itself is known to be well-defined for all $\ais$ in a neighbourhood $\mathcal B\subset (\Cx)^n$ of the point $(a,\dots,a)\in (\Cx)^n$, so these singularities must be an artifact of the choice of basis. Our goal is thus to construct a new basis, say \be\left( \ZZ A\right)_{A\subseteq \range 1 n}\ee of $\V_\ais$ in which the limit makes sense: that is, a basis in terms of which the matrix elements are regular for all $\ais\in \mathcal B$. We will work initially with $\eps\neq 0$ and define a new basis explicitly in term of the old:
\be \ZZ A = \sum_{B\subseteq\range 1 n}  \VV B M^B{}_A,\ee
where the change-of-basis matrix $M=M(\eps;\alpha_1,\dots,\alpha_n,a,q)$ must of course be non-singular and invertible for $\eps\neq 0$. In the limit $\eps\to 0$ we will demand that $M$ becomes non-invertible (and its inverse singular) in such a way as to absorb all the singularities of the matrix elements in the $\VV{}$-basis. 
That is, writing
\be x\on \VV A = \sum_{B\subseteq\range 1 n} \VV B \rho_{V}(x)^B{}_A, \qquad 
    x\on \ZZ A = \sum_{B\subseteq \range 1 n} \ZZ B \rho_{Z}(x)^B{}_A \ee
where $x\in\uqslth$, we demand that all the matrix elements
\be \rho_{Z}(x)^B{}_A:= \left((M^{-1})\cdot \rho_{ V}(x)\cdot M\right)^B{}_A\label{goodbasis}\ee
remain regular. The degeneracy in $M$ as $\eps\to 0$ will allow $\rho_{ Z}(\phi^{\pm}(u))$ to develop nontrivial Jordan blocks, even though $\rho_V(\phi^\pm(u))$ is diagonal (and not, itself, singular in the limit). 

With this outline in place, let us proceed to define the basis $\ZZ A$ and show it has the required properties. First,  associate to each subset $A\subseteq\range 1 n$ a path in the plane $\mathbb R^2$, its \emph{lattice path}, 
as follows: start at the origin and, for each integer $j\in\range 1 n$ in turn, extend the path by adding a line segment of unit length, in the direction $(0,1)$ if $j\in A$ and $(1,0)$ if $j\notin A$. All these lattice paths have length $n$ and no two distinct subsets of $\range 1 n$ yield the same lattice path. An example is shown in figure \ref{staircasefig}.

Suppose $A\subseteq \range 1 n$ has $m$ elements. We shall write $(A_i)_{\ir m}$ for the $m$-tuple consisting of the elements of $A$ in increasing order, \ie $A= \{A_1,A_2,\dots, A_{m}\}$, $A_1<A_2<\dots<A_{m}$. Define $\lambdaof A$ to be the partition  (\ie weakly decreasing sequence of non-negative integers, which it is convenient here to allow to end with one or more zeros) whose $m$ parts are given by
\be \lambdaof{A}_{m+1-i} := A_i - i,\qquad \ir m.\ee
The Ferrers diagram of $\lambdaof A$ is the subset of the rectangle with vertices $(0,0)$, $(n-m,0)$, $(n-m,m)$ and $(0,m)$ that lies above the lattice path of $A$. Let also $\comp A := \range 1 n \setminus A$, and  
observe that 
 $|\lambdaof A| + |\lambdaof {\comp A}| = |A|\cdot |\comp A|$, where $|\lambda|:=\sum_i\lambda_i$.  
We define the \emph{Schur polynomial} $s_\lambda(z_1,\dots ,z_m)$ associated to any partition $\lambda$ with $m$ parts by 
\be s_\lambda(z_1,\dots ,z_m) = \frac{\det(z_{j}^{\lambda_{m+1-i}+i-1})_{\ijr m} }
                                     {\det(z_{j}^{        i-1}  )_{\ijr m} } \label{schurdef}.\ee

\begin{figure}
\begin{center}
\tikz[baseline=0pt,scale=.75]{\def \n{9}; \foreach \x in {0,...,\n}
                    { \draw (\x,0)--(\x,\n-\x) ;
                      \draw (0,\x)--(\n-\x,\x) ;  }
\draw[very thick] (0,0) -- node[above,sloped] {$1$} 
              (0,1) -- node[below]  {$2$}  
              (1,1) -- node[above,sloped]  {$3$} 
              (1,2) -- node[below]  {$4$} 
              (2,2) -- node[below]  {$5$}
              (3,2) -- node[above,sloped]  {$6$} 
              (3,3) -- node[above,sloped]  {$7$} 
              (3,4) -- node[below]  {$8$} 
              (4,4) -- node[below]  {$9$} (5,4);}
\end{center}

\be\nn
\tikz[baseline=0pt,scale=.75,x={(0cm,1cm)},y={(1cm,0cm)}]{
\path [fill=gray!50] (0,0) -- +(1,0) -- +(1,1) -- +(2,1) -- +(2,3) -- +(3,3) -- +(4,3) -- +(4,5) -- +(4,0)--cycle;
\path (0,0) -- node[left] {$0)$} +(1,0) -- node[left] {$1,$} +(2,0) 
            -- node[left] {$3,$} +(3,0) -- node[left] {$(3,$} +(4,0);
\draw[gray] (0,0) grid +(4,5);
\draw [very thick]  (0,0) -- +(1,0) -- +(1,1) -- +(2,1) -- +(2,3) -- +(3,3) -- +(4,3) -- +(4,5);
}\qquad
\tikz[baseline=0pt,scale=.75]{
\path [fill=gray!50] (0,0) -- +(1,0) -- +(1,1) -- +(2,1) -- +(2,3) -- +(3,3) -- +(4,3) -- +(4,5) -- +(0,5)--cycle;
\path (0,0)+(0,5) -- node[left]  {$(4,$} +(0,4) -- node[left]  {$\phantom ( 4,$} +(0,3)
            -- node[left]  {$\phantom ( 2,$} +(0,2) -- node[left]  {$\phantom ( 2,$} +(0,1)
            -- node[left]  {$1)$} +(0,0);
\draw[gray] (0,0) grid +(4,5);
\draw [very thick] (0,0) -- +(1,0) -- +(1,1) -- +(2,1) -- +(2,3) -- +(3,3) -- +(4,3) -- +(4,5);
}\ee
\caption{\label{staircasefig} With $n=9$, the lattice path of the subset $\{1,3,6,7\}\subset \range 1 9$ and the Ferrers diagrams of the partitions $\lambdaof {\{1,3,6,7\}} = (3,3,1,0)$ and $\lambdaof {\overline{\{1,3,6,7\}}} = \lambdaof {\{2,4,5,8,9\}}=(4,4,2,2,1)$.}
\end{figure}

\begin{definition}\label{Zdef}
For all $A\subseteq \range 1 n$, let
\be \ZZ A  = \sum_{\substack{B\subseteq \range 1 n \\  |B|=|A|}} \frac{\det(z_{B_j}^{A_i-1})_{\ijr |B|} }
                                                                  {\det(z_{B_j}^{  i-1})_{\ijr |B|} }  \VV B  
         = \sum_{\substack{B\subseteq \range 1 n \\  |B|=|A|}} s_{\lambdaof A}(z_B) \VV B ,\ee
where, for any $B\subseteq\range 1 n$, we adopt the shorthand $z_B = (z_{B_1},\dots, z_{B_{|B|}})$.
\end{definition}
We need to be able to invert this change of basis. To that end, if $B\subseteq\range 1 n$ with $k=|B|$, define the permutation $w_B\in S_n$ to be 
\be w_B={\small \left(\begin{matrix} 1  &  \dots & k&k+1 & \dots &n \\
                    \down &  &  \down&  \down& & \down \\ 
                         B_1  & \dots & B_k & \comp B_1 & \dots & \comp B_{n-k} \end{matrix}\right)}\ee
and note that 
\be \sgn(w_B)= \prod_{\substack{i\in B\\j\notin B}} \sign(j-i) = (-1)^{|\lambdaof B|}\label{epssigneqn}\ee
and hence $\sgn(w_B)\sgn(w_{\comp B}) = (-1)^{|B|\cdot |\comp B|}$. 
Given any matrix $(m_{i,j})_{\ijr n}$, one has 
\bea \det (m_{i,j})_{\ijr n}  
&=&  \sum_{\substack{B\subseteq \range 1 n \\  |B|=k}}   \sgn(w_B)  \det(m_{i,B_j})_{\ijr k} \, \det(m_{k+i,\comp B_j})_{\ijr{n-k}}\label{detsplit}\eea 
for all $0\leq k \leq n$.
Let $\Delta = \det (z_j^{i-1})_{\ijr n} = \prod_{1\leq i<j\leq n} (z_j - z_i)$ denote the Vandermonde determinant in the variables $z_1,\dots, z_n$. Clearly, for any $n$-tuple $(I_1,I_2,\dots, I_n)\in \range 1 n ^n$, 
\be \det (z_{j}^{I_i-1})_{\ijr n}  = \begin{cases} 
\sgn(\sigma) \Delta & \text{if } \exists \,\sigma \in S_n \text{ s.t. } I_i = \sigma i\,\,\forall i \in\range 1 n\\          
0 & \text{otherwise.}\end{cases}\ee
If $A$ and $C$ are both subsets of $\range 1 n$ with $k$ elements, let $A \concat \comp C\in \range 1 n^n$ be the $n$-tuple obtained by concatenating $(A_i)_{\ir k}$ and $(\comp C_i)_{\ir {n-k}}$; then in particular
\be \det ( z_j^{(A \concat \comp C)_i-1} )_{\ijr n}
=  \sgn(w_A)  \Delta  \, \delta_{A,C} =  \sgn(w_C) \Delta  \, \delta_{A,C} ,\ee
where of course $\delta_{A,C}$ is defined to be 1 if $A=C$ and 0 otherwise. 
But 
by (\ref{detsplit}), 
\be \det ( z_j^{(A \concat \comp C)_i-1 })_{\ijr n}
   = \sum_{\substack{B\subseteq \range 1 n \\  |B|=k}}  \sgn(w_B) 
\det(z^{A_i-1}_{B_j})_{\ijr k} \, \det(z^{\comp C_i-1}_{\comp B_j})_{\ijr{n-k}}.\ee
Comparing these two expressions, and noting that for any subset $B \subseteq \range 1 n$ with $k$ elements the Vandermonde determinant can be factored as
\bea\nn \Delta &=& \prod_{\substack{ i\in B \\ j \notin B}} (z_i - z_j) \sign(i-j) 
\prod_{\substack{ i<j \\ i,j \in B}} (z_j - z_i) \prod_{\substack{ i<j \\ i,j \notin B}} (z_j - z_i)  \\
  &=& (-1)^{|B|\cdot|\comp B|} \sgn(w_B)  \prod_{\substack{ i\in B \\ j \notin B}} (z_i - z_j)  \det (z^{i-1}_{B_j})_{\ijr k} \det(z^{i-1}_{\comp B_j})_{\ijr{n-k}}\label{vdmfactored}, 
\eea
we find
\bea \delta_{A,C} 
&=& (-1)^{|C|\cdot |\comp C|} \sum_{\substack{B\subseteq \range 1 n \\  |B|=k}}  
\frac{\det(z^{A_i-1}_{B_j})_{\ijr k}}{\det(z^{i-1}_{B_j})_{\ijr k}} \, 
\frac{\det(z^{\comp C_i-1}_{\comp B_j})_{\ijr{n-k}}}{\det(z^{i-1}_{\comp B_j})_{\ijr{n-k}}} 
\frac{\sgn(w_C)}{ \prod_{\substack{ i\in B \\ j \notin B}} (z_i - z_j)}.\label{deltaid}\nn\\
&=& \label{deltarel} \sum_{\substack{B\subseteq \range 1 n \\  |B|=k}}  
s_{\lambdaof A}(z_B) 
s_{\lambdaof {\comp C}}(z_{\comp B})
\frac{\sgn(w_{\comp C})}{ \prod_{\substack{ i\in B \\ j \notin B}} (z_i - z_j)}\eea
The above result allows us to recover the $\VV B$ in terms of the $\ZZ A$, for we have established
\begin{proposition} For all $B\subseteq \range 1 n$,   \label{VtoZ}
\be\nn \VV B = \frac{1}{ \prod_{\substack{ i\in B \\ j \notin B}} (z_i - z_j)}
 \sum_{\substack{C\subseteq \range 1 n \\  |C|=|B|}} \sgn(w_{\comp C}) s_{\lambdaof {\comp C}}(z_{\comp B}) \ZZ C.\ee
\end{proposition}

For all $0\leq k\leq n$, let $\Lambda_k$ denote the ring of symmetric polynomial functions of $k$ variables with complex coefficients. When we wish to specify that the variables are $z_1,\dots, z_k$ we will write $\mathbb C[z_1,\dots,z_k]^{S_k}$ (clearly $\Lambda_k\cong\mathbb C[z_1,\dots,z_k]^{S_k}$ as rings) and, more generally, if $B\subseteq \range 1 n$ has $k$ elements and $p\in\Lambda_{k}$ we will continue to write $p(z_B)$ as a shorthand for $p(z_{B_1},\dots,z_{B_{k}})$. The first important property of the basis $\ZZ A$ is then contained in the following

\begin{proposition}\label{innerprop} Let $0\leq k \leq n$. For any $p\in\Lambda_k$ and any $\tilde p \in \Lambda_{n-k}$, \be\sum_{\substack{B\subseteq \range 1 n\\ |B|=k}} p(z_B) \VV B\qquad\text{and}\qquad\sum_{\substack{B\subseteq \range 1 n\\ |B|=k}} \tilde p(z_{\comp B}) \VV B\ee are both linear combinations of the $\ZZ A$ with coefficients in $\mathbb C[z_1,\dots z_n]^{S_n}$.
\end{proposition}
\begin{proof}
Since the Schur polynomials form a $\mathbb C$-basis for $\Lambda_k$ \cite{MacDonald}, it suffices in the first sum to consider $p=s_\mu$ for all partitions $\mu$ having $k$ parts. (We allow the possibility that some of the parts are zero, as noted above.) And then indeed, using proposition \ref{VtoZ} and (\ref{schurdef}), one has that the first sum is equal to 
\bea   \label{smuB} &&
\sum_{\substack{B\subseteq \range 1 n\\ |B|=k}}  
                              s_{\mu}(z_B)
\sum_{\substack{C\subseteq \range 1 n \\  |C|=|B|}} 
\frac{\sgn(w_{\comp C}) s_{\lambdaof {\comp C}}(z_{\comp B})}{ \prod_{\substack{ i\in B \\ j \notin B}} (z_i - z_j)} \ZZ C\\
&=& \sum_{\substack{C\subseteq \range 1 n \\  |C|=k}} \sgn(w_{\comp C}) \ZZ C \sum_{\substack{B\subseteq \range 1 n\\ |B|=k}}  
                              \frac{\det(z_{B_j}^{\mu_{k+1-i}+i-1} )_{\ijr k} }
                                   {\det(z_{B_j}^{        i-1})_{\ijr k} }
\frac{\det(z^{\comp C_i-1}_{\comp B_j})_{\ijr{n-k}}}{\det(z^{i-1}_{\comp B_j})_{\ijr{n-k}}} 
\frac{1}{ \prod_{\substack{ i\in B \\ j \notin B}} (z_i - z_j)}\nn .\eea
This is in the right form to apply (\ref{vdmfactored}) and then (\ref{detsplit}); on doing so one has simply 
\be \sum_{\substack{B\subseteq \range 1 n\\ |B|=k}} s_\mu(z_B) \VV B 
   = \sum_{\substack{C\subseteq \range 1 n \\  |C|=k}} \sgn(w_{C}) \ZZ C \frac{1}{\Delta} 
      \det(z_j^{\tau_i-1})_{\ijr n}\label{leftin}\ee
where $(\tau_i)_{\ir n}$ is the sequence formed by concatenating $(\mu_{k+1-i}+i)_{\ir k}$ and $(\comp C_i)_{\ir {n-k}}$ (whose entries, unlike those of $A \concat \comp C$ above, need not be elements of $\range 1 n$, since the parts of $\mu$ can be arbitrarily large). For any $C$ such that the $\tau_i$ are not all distinct, the coefficient of $\ZZ C$ vanishes. 
Otherwise, $(\tau_i)_{\ir n}$ is a permutation of the strictly increasing sequence $(\rho_{n+1-i}+i)_{\ir n}$ defined by some partition $\rho$ with $n$ parts and, up to a sign, the coefficient of $\ZZ C$ is $s_\rho \in \mathbb C[z_1,\dots z_n]^{S_n}$, as required.

Likewise, in the second sum it suffices to consider the Schur polynomials, $\tilde p = s_\mu$, for partitions $\mu$ with $n-k$ parts. Let $c_{\mu\nu}^\rho$ be the Littlewood-Richardson coefficients, \ie $s_\mu s_\nu = \sum_{\rho} c_{\mu\nu}^\rho s_\rho$. Then
\bea   \sum_{\substack{B\subseteq \range 1 n\\ |B|=k}} s_\mu(z_{\comp B}) \VV B
&=& \sum_{\substack{C\subseteq \range 1 n \\  |C|=k}}
                  \sgn(w_{\comp C})  \ZZ C \sum_\rho  c_{\mu \lambdaof {\comp C}}^\rho     
\sum_{\substack{B\subseteq \range 1 n\\ |B|=k}} 
\frac{s_{\rho}(z_{\comp B})}{ \prod_{\substack{ i\in B \\ j \notin B}} (z_i - z_j)} \eea
and here we can use similar manipulations to those above -- that is,
\bea \sum_{\substack{B\subseteq \range 1 n\\ |B|=k}} 
\frac{s_{\rho}(z_{\comp B})}{ \prod_{\substack{ i\in B \\ j \notin B}} (z_i - z_j)} 
&=& 
\sum_{\substack{B\subseteq \range 1 n\\ |B|=k}}  
  \frac{\det(z_{B_j}^{i-1} )_{\ijr k} }                  {\det(z_{B_j}^{i-1})_{\ijr k} }
  \frac{\det(z^{ \rho_{n-k+1-i}+i-1}_{\comp B_j})_{\ijr{n-k}}}{\det(z^{i-1}_{\comp B_j})_{\ijr{n-k}}} 
  \frac{1}{ \prod_{\substack{ i\in B \\ j \notin B}} (z_i - z_j)}\nn\\ 
&=& (-1)^{|B|\cdot |\comp B|} \frac{1}{\Delta} \det(z_j^{\tau_i-1})_{\ijr n}\label{rightin}\eea
where now $(\tau_i)_{\ir n}$ is the concatenation of $(1,\dots, k)$ with $(\rho_{n-k+1-i}+i)_{\ir{n-k}}$; and the result is once more a symmetric polynomial in $z_1,\dots z_n$. \finproof
\end{proof}

We will now argue that the $\ZZ A$ are a good basis in the $\eps\to 0$ limit in the sense discussed above, c.f. (\ref{goodbasis}). First, we have

\begin{proposition}\label{loweringprop} Let $m\geq 0$ and $\tau=(\tau_1,\dots,\tau_m)\in \Z^m$. Then 
\be x^-_{\tau_m}\on \dots \on x^-_{\tau_1} \on \VV\emptyset 
   = \sum_{\substack{B\subseteq \range 1 n\\ |B|=m}} R_{\tau} (a_B) \VV B \label{loweringmodeswrtVs}\ee
where $R_\tau$ is a symmetric Laurent polynomial \footnote{\ie a polynomial in the variables and their inverses}  in $m$ variables defined by
\be R_\tau(y_1,\dots,y_m) = \sum_{\sigma\in S_m}  y^{\tau_1}_{\sigma 1} \dots y^{\tau_m}_{\sigma m}
      \prod_{i<j} \frac{qy_{\sigma j}-q^{-1} y_{\sigma i}}{y_{\sigma j}-y_{\sigma i}}. \label{Rdef}\ee
\end{proposition}
\begin{remark} When in addition $\tau_1\geq\dots\geq\tau_m\geq 0$, so that $\tau$ is a partition, $R_\tau$ is proportional to the  Hall-Littlewood polynomial $Q_\tau$; for the definition see \cite{MacDonald}, especially (\textsc{III}.2.14). \end{remark}
\begin{proof}
First observe that, for any $0\leq k\leq m$,
\bea  R_\tau(y_1,\dots,y_m)
&=& \sum_{\substack{B\subseteq\range 1 m\\|B|=k}} \sum_{\rho\in S_k}\sum_{\mu\in S_{m-k}}
      y^{\tau_1}_{B_{\rho 1}} \dots y^{\tau_k}_{B_{\rho k}} y^{\tau_{k+1}}_{\comp B_{\mu 1}} \dots y^{\tau_{m}}_{\comp B_{\mu m}}\\
&&\quad  \prod_{\substack{i,j\in B\\i<j}}         \frac{qy_{\rho j}-q^{-1} y_{\rho i}}{y_{\rho j}-y_{\rho i}} 
          \prod_{\substack{i,j\in \comp B\\i<j}}    \frac{qy_{\mu j}-q^{-1} y_{\mu i}}{y_{\mu j}-y_{\mu i}} 
          \prod_{\substack{i\in B\\ j \in \comp B}} \frac{qy_{j}-q^{-1} y_{i}}{y_{j}-y_{i}}\nn\\
 &=& \sum_{\substack{B\subseteq\range 1 m\\|B|=k}} 
                        R_{(\tau_{1},\dots,\tau_{k})}(y_B) R_{(\tau_{k+1},\dots,\tau_{m}})(y_{\comp B})
                                  \prod_{\substack{i\in B\\j\in \comp B}} \frac{qy_j-q^{-1}y_i}{y_j-y_i}.\nn\eea 
(This can be seen as a $q$-deformed analog of (\ref{detsplit}).) Now we will establish (\ref{loweringmodeswrtVs}) by induction on $m$. 
It is trivial for $m=0$. Suppose that it holds for some $m\geq 0$. Let  $\tau=(\tau_1,\dots,\tau_{m+1})\in \Z^{m+1}$. Then
\bea x^-_{\tau_{m+1}} \on \dots \on x^-_{\tau_1} \on \VV\emptyset 
    &=& x^-_{\tau_{m+1}} \on \sum_{\substack{B\subseteq \range 1 n\\ |B|=m}} R_{(\tau_1,\dots,\tau_m)} (a_B) \VV B \\
    &=& \sum_{\substack{B\subseteq \range 1 n\\ |B|=m}} 
                   \sum_{j\notin B} \VV {B\cup \{j\}} a_j^{\tau_{m+1}} R_{(\tau_1,\dots,\tau_m)} (a_B)   
                                      \prod_{k \in B}  \frac{a_k q - a_j q^{-1}}{a_k - a_j} \nn\\
    &=& \sum_{\substack{B\subseteq \range 1 n\\ |B|=m+1}} \VV B 
                   \sum_{j\in B} a_j^{\tau_{m+1}} R_{(\tau_1,\dots,\tau_m)} (a_{B\setminus\{j\}})
                               \prod_{\substack{k\neq j \\k \in B}}  \frac{a_k q - a_j q^{-1}}{a_k - a_j}\nn\eea  
and therefore (\ref{loweringmodeswrtVs}) holds also for $m+1$ by virtue of the observation above (with $k=1$). 

Finally, note that $R_\tau(y_1,\dots,y_m) \prod_{i<j} (y_j-y_i)=\sum_{\sigma\in S_m} \sgn(\sigma) y^{\tau_1}_{\sigma 1} \dots y^{\tau_m}_{\sigma m}\prod_{i<j} \left(qy_{\sigma j}-q^{-1} y_{\sigma i}\right)$, which is a skew-symmetric Laurent polynomial and therefore divisible by $\prod_{i<j} (y_j-y_i)$ in the ring of Laurent polynomials in $y_1,\dots,y_m$. The quotient, $R_\tau$, is thus indeed a symmetric Laurent polynomial.
\finproof\end{proof}
 
Now, on setting $a_i=a+z_i=a+\eps \alpha_i$, we can expand (\ref{loweringmodeswrtVs}) in powers of $\eps$ to find that
for all $m\geq 0$ and all $\tau=(\tau_1,\dots,\tau_m)\in \Z^m$, 
\be x^-_{\tau_m}\on \dots \on x^-_{\tau_1} \on \VV\emptyset 
   = \sum_{\substack{B\subseteq \range 1 n\\ |B|=m}} R'_{\tau} (z_B) \VV B \ee
for some symmetric power series $R'_\tau(z_1,\dots,z_m)\in \mathbb C[[z_1,\dots,z_m]]^{S_m}$. The right-hand side is trivially zero for $m>n$. For all $0\leq m \leq n$ we can regard $R'_\tau$ as an infinite sum of (say, homogeneous) elements of $\Lambda_m$. By applying proposition \ref{innerprop} term-by-term in this sum, we conclude that, for all $\tau=(\tau_1,\dots,\tau_m)\in \Z^m$, 
\be x^-_{\tau_m}\on \dots \on x^-_{\tau_1} \on \VV\emptyset 
   = \sum_{A\subseteq \range 1 n} c_{A,\tau}(z_1,\dots z_n) \ZZ A \label{catsum}\ee
for certain symmetric power series $c_{A,\tau}(z_1,\dots z_n)\in\mathbb C[[z_1,\dots,z_n]]^{S_n}$.
On the other hand, it follows from proposition \ref{Prop:Stdhighest} that, for all $\ais\in \cal B$, these vectors 
span $\V_\ais$ (as a $\mathbb C$-module).
Since we have shown that each of them is a linear combination of the $\ZZ A$, with coefficients that remain finite in the limit $\eps\to 0$, we have indeed shown that
\begin{proposition} The vectors $\{\ZZ A:A\subseteq \range 1 n\}$ of definition (\ref{Zdef}) constitute a well-defined basis of $\V_\ais$ for all $\ais\in \mathcal B$. \end{proposition}
Clearly, in the coincident limit only those $A\subseteq \range 1 n$ in the sum (\ref{catsum}) for which $c_{A,\tau}$ is proportional to the identity polynomial survive. Note that one could in principle also construct a well-defined basis from vectors of the form $x^-_{\tau_m}\on \dots \on x^-_{\tau_1} \on \VV\emptyset $ directly. However, the $\ZZ{A}$ basis as defined above also has the important property that $\phi^\pm(u)$ acts upper-triangularly, in a sense to which we now turn.

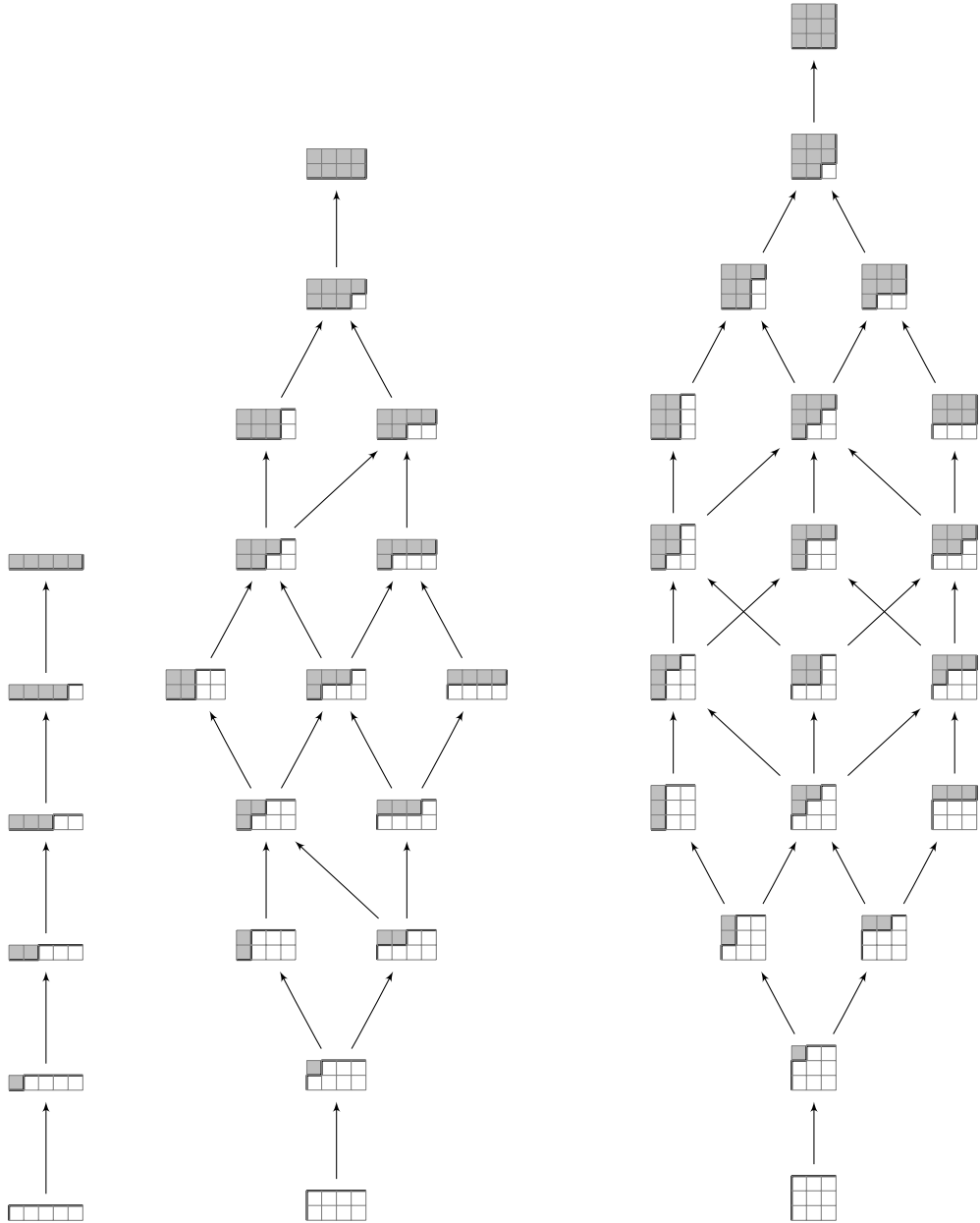
\begin{figure}
\be\nn
\begin{tikzpicture}[>=latex',line join=bevel,rotate=180]
\node (A1) at (18bp,207bp) [draw,draw=none] {$\begin{tikzpicture}[scale=.2] \path[fill=gray!50] (0,0)  -- ++(1,0)-- ++(0,1) -- ++(1,0) -- ++(1,0) -- ++(1,0) -- ++(1,0) -- (0,1) -- cycle ;  \draw[thick] (0,0)  -- ++(1,0)-- ++(0,1) -- ++(1,0) -- ++(1,0) -- ++(1,0) -- ++(1,0); \draw[gray] (0,0) grid (5,1); \end{tikzpicture}$};
  \node (A0) at (18bp,257bp) [draw,draw=none] {$\begin{tikzpicture}[scale=.2] \path[fill=gray!50] (0,0) -- ++(0,1) -- ++(1,0) -- ++(1,0) -- ++(1,0) -- ++(1,0) -- ++(1,0) -- (0,1) -- cycle ;  \draw[thick] (0,0) -- ++(0,1) -- ++(1,0) -- ++(1,0) -- ++(1,0) -- ++(1,0) -- ++(1,0); \draw[gray] (0,0) grid (5,1); \end{tikzpicture}$};
  \node (A3) at (18bp,107bp) [draw,draw=none] {$\begin{tikzpicture}[scale=.2] \path[fill=gray!50] (0,0)  -- ++(1,0) -- ++(1,0) -- ++(1,0)-- ++(0,1) -- ++(1,0) -- ++(1,0) -- (0,1) -- cycle ;  \draw[thick] (0,0)  -- ++(1,0) -- ++(1,0) -- ++(1,0)-- ++(0,1) -- ++(1,0) -- ++(1,0); \draw[gray] (0,0) grid (5,1); \end{tikzpicture}$};
  \node (A2) at (18bp,157bp) [draw,draw=none] {$\begin{tikzpicture}[scale=.2] \path[fill=gray!50] (0,0)  -- ++(1,0) -- ++(1,0)-- ++(0,1) -- ++(1,0) -- ++(1,0) -- ++(1,0) -- (0,1) -- cycle ;  \draw[thick] (0,0)  -- ++(1,0) -- ++(1,0)-- ++(0,1) -- ++(1,0) -- ++(1,0) -- ++(1,0); \draw[gray] (0,0) grid (5,1); \end{tikzpicture}$};
  \node (A5) at (18bp,7bp) [draw,draw=none] {$\begin{tikzpicture}[scale=.2] \path[fill=gray!50] (0,0)  -- ++(1,0) -- ++(1,0) -- ++(1,0) -- ++(1,0) -- ++(1,0)-- ++(0,1) -- (0,1) -- cycle ;  \draw[thick] (0,0)  -- ++(1,0) -- ++(1,0) -- ++(1,0) -- ++(1,0) -- ++(1,0)-- ++(0,1); \draw[gray] (0,0) grid (5,1); \end{tikzpicture}$};
  \node (A4) at (18bp,57bp) [draw,draw=none] {$\begin{tikzpicture}[scale=.2] \path[fill=gray!50] (0,0)  -- ++(1,0) -- ++(1,0) -- ++(1,0) -- ++(1,0)-- ++(0,1) -- ++(1,0) -- (0,1) -- cycle ;  \draw[thick] (0,0)  -- ++(1,0) -- ++(1,0) -- ++(1,0) -- ++(1,0)-- ++(0,1) -- ++(1,0); \draw[gray] (0,0) grid (5,1); \end{tikzpicture}$};
  \draw [->] (A0) ..controls (18bp,243bp) and (18bp,233bp)  .. (A1);
  \draw [->] (A2) ..controls (18bp,143bp) and (18bp,133bp)  .. (A3);
  \draw [->] (A1) ..controls (18bp,193bp) and (18bp,183bp)  .. (A2);
  \draw [->] (A3) ..controls (18bp,93bp) and (18bp,83bp)  .. (A4);
  \draw [->] (A4) ..controls (18bp,43bp) and (18bp,33bp)  .. (A5);
\end{tikzpicture}
\quad\quad
\begin{tikzpicture}[>=latex',line join=bevel,rotate=180]
\node (A15) at (45bp,157bp) [draw,draw=none] {$\begin{tikzpicture}[scale=.2] \path[fill=gray!50] (0,0)  -- ++(1,0)-- ++(0,1) -- ++(1,0) -- ++(1,0) -- ++(1,0)-- ++(0,1) -- (0,2) -- cycle ;  \draw[thick] (0,0)  -- ++(1,0)-- ++(0,1) -- ++(1,0) -- ++(1,0) -- ++(1,0)-- ++(0,1); \draw[gray] (0,0) grid (4,2); \end{tikzpicture}$};
  \node (A14) at (72bp,207bp) [draw,draw=none] {$\begin{tikzpicture}[scale=.2] \path[fill=gray!50] (0,0)  -- ++(1,0)-- ++(0,1) -- ++(1,0) -- ++(1,0)-- ++(0,1) -- ++(1,0) -- (0,2) -- cycle ;  \draw[thick] (0,0)  -- ++(1,0)-- ++(0,1) -- ++(1,0) -- ++(1,0)-- ++(0,1) -- ++(1,0); \draw[gray] (0,0) grid (4,2); \end{tikzpicture}$};
  \node (A23) at (126bp,207bp) [draw,draw=none] {$\begin{tikzpicture}[scale=.2] \path[fill=gray!50] (0,0)  -- ++(1,0) -- ++(1,0)-- ++(0,1)-- ++(0,1) -- ++(1,0) -- ++(1,0) -- (0,2) -- cycle ;  \draw[thick] (0,0)  -- ++(1,0) -- ++(1,0)-- ++(0,1)-- ++(0,1) -- ++(1,0) -- ++(1,0); \draw[gray] (0,0) grid (4,2); \end{tikzpicture}$};
  \node (A24) at (99bp,157bp) [draw,draw=none] {$\begin{tikzpicture}[scale=.2] \path[fill=gray!50] (0,0)  -- ++(1,0) -- ++(1,0)-- ++(0,1) -- ++(1,0)-- ++(0,1) -- ++(1,0) -- (0,2) -- cycle ;  \draw[thick] (0,0)  -- ++(1,0) -- ++(1,0)-- ++(0,1) -- ++(1,0)-- ++(0,1) -- ++(1,0); \draw[gray] (0,0) grid (4,2); \end{tikzpicture}$};
  \node (A25) at (45bp,107bp) [draw,draw=none] {$\begin{tikzpicture}[scale=.2] \path[fill=gray!50] (0,0)  -- ++(1,0) -- ++(1,0)-- ++(0,1) -- ++(1,0) -- ++(1,0)-- ++(0,1) -- (0,2) -- cycle ;  \draw[thick] (0,0)  -- ++(1,0) -- ++(1,0)-- ++(0,1) -- ++(1,0) -- ++(1,0)-- ++(0,1); \draw[gray] (0,0) grid (4,2); \end{tikzpicture}$};
  \node (A13) at (99bp,257bp) [draw,draw=none] {$\begin{tikzpicture}[scale=.2] \path[fill=gray!50] (0,0)  -- ++(1,0)-- ++(0,1) -- ++(1,0)-- ++(0,1) -- ++(1,0) -- ++(1,0) -- (0,2) -- cycle ;  \draw[thick] (0,0)  -- ++(1,0)-- ++(0,1) -- ++(1,0)-- ++(0,1) -- ++(1,0) -- ++(1,0); \draw[gray] (0,0) grid (4,2); \end{tikzpicture}$};
  \node (A12) at (99bp,307bp) [draw,draw=none] {$\begin{tikzpicture}[scale=.2] \path[fill=gray!50] (0,0)  -- ++(1,0)-- ++(0,1)-- ++(0,1) -- ++(1,0) -- ++(1,0) -- ++(1,0) -- (0,2) -- cycle ;  \draw[thick] (0,0)  -- ++(1,0)-- ++(0,1)-- ++(0,1) -- ++(1,0) -- ++(1,0) -- ++(1,0); \draw[gray] (0,0) grid (4,2); \end{tikzpicture}$};
  \node (A02) at (72bp,357bp) [draw,draw=none] {$\begin{tikzpicture}[scale=.2] \path[fill=gray!50] (0,0) -- ++(0,1) -- ++(1,0)-- ++(0,1) -- ++(1,0) -- ++(1,0) -- ++(1,0) -- (0,2) -- cycle ;  \draw[thick] (0,0) -- ++(0,1) -- ++(1,0)-- ++(0,1) -- ++(1,0) -- ++(1,0) -- ++(1,0); \draw[gray] (0,0) grid (4,2); \end{tikzpicture}$};
  \node (A03) at (45bp,307bp) [draw,draw=none] {$\begin{tikzpicture}[scale=.2] \path[fill=gray!50] (0,0) -- ++(0,1) -- ++(1,0) -- ++(1,0)-- ++(0,1) -- ++(1,0) -- ++(1,0) -- (0,2) -- cycle ;  \draw[thick] (0,0) -- ++(0,1) -- ++(1,0) -- ++(1,0)-- ++(0,1) -- ++(1,0) -- ++(1,0); \draw[gray] (0,0) grid (4,2); \end{tikzpicture}$};
  \node (A01) at (72bp,407bp) [draw,draw=none] {$\begin{tikzpicture}[scale=.2] \path[fill=gray!50] (0,0) -- ++(0,1)-- ++(0,1) -- ++(1,0) -- ++(1,0) -- ++(1,0) -- ++(1,0) -- (0,2) -- cycle ;  \draw[thick] (0,0) -- ++(0,1)-- ++(0,1) -- ++(1,0) -- ++(1,0) -- ++(1,0) -- ++(1,0); \draw[gray] (0,0) grid (4,2); \end{tikzpicture}$};
  \node (A04) at (45bp,257bp) [draw,draw=none] {$\begin{tikzpicture}[scale=.2] \path[fill=gray!50] (0,0) -- ++(0,1) -- ++(1,0) -- ++(1,0) -- ++(1,0)-- ++(0,1) -- ++(1,0) -- (0,2) -- cycle ;  \draw[thick] (0,0) -- ++(0,1) -- ++(1,0) -- ++(1,0) -- ++(1,0)-- ++(0,1) -- ++(1,0); \draw[gray] (0,0) grid (4,2); \end{tikzpicture}$};
  \node (A05) at (18bp,207bp) [draw,draw=none] {$\begin{tikzpicture}[scale=.2] \path[fill=gray!50] (0,0) -- ++(0,1) -- ++(1,0) -- ++(1,0) -- ++(1,0) -- ++(1,0)-- ++(0,1) -- (0,2) -- cycle ;  \draw[thick] (0,0) -- ++(0,1) -- ++(1,0) -- ++(1,0) -- ++(1,0) -- ++(1,0)-- ++(0,1); \draw[gray] (0,0) grid (4,2); \end{tikzpicture}$};
  \node (A34) at (99bp,107bp) [draw,draw=none] {$\begin{tikzpicture}[scale=.2] \path[fill=gray!50] (0,0)  -- ++(1,0) -- ++(1,0) -- ++(1,0)-- ++(0,1)-- ++(0,1) -- ++(1,0) -- (0,2) -- cycle ;  \draw[thick] (0,0)  -- ++(1,0) -- ++(1,0) -- ++(1,0)-- ++(0,1)-- ++(0,1) -- ++(1,0); \draw[gray] (0,0) grid (4,2); \end{tikzpicture}$};
  \node (A35) at (72bp,57bp) [draw,draw=none] {$\begin{tikzpicture}[scale=.2] \path[fill=gray!50] (0,0)  -- ++(1,0) -- ++(1,0) -- ++(1,0)-- ++(0,1) -- ++(1,0)-- ++(0,1) -- (0,2) -- cycle ;  \draw[thick] (0,0)  -- ++(1,0) -- ++(1,0) -- ++(1,0)-- ++(0,1) -- ++(1,0)-- ++(0,1); \draw[gray] (0,0) grid (4,2); \end{tikzpicture}$};
  \node (A45) at (72bp,7bp) [draw,draw=none] {$\begin{tikzpicture}[scale=.2] \path[fill=gray!50] (0,0)  -- ++(1,0) -- ++(1,0) -- ++(1,0) -- ++(1,0)-- ++(0,1)-- ++(0,1) -- (0,2) -- cycle ;  \draw[thick] (0,0)  -- ++(1,0) -- ++(1,0) -- ++(1,0) -- ++(1,0)-- ++(0,1)-- ++(0,1); \draw[gray] (0,0) grid (4,2); \end{tikzpicture}$};
  \draw [->] (A35) ..controls (72bp,43bp) and (72bp,33bp)  .. (A45);
  \draw [->] (A13) ..controls (107bp,243bp) and (113bp,232bp)  .. (A23);
  \draw [->] (A24) ..controls (99bp,143bp) and (99bp,133bp)  .. (A34);
  \draw [->] (A03) ..controls (45bp,293bp) and (45bp,283bp)  .. (A04);
  \draw [->] (A34) ..controls (91bp,93bp) and (85bp,82bp)  .. (A35);
  \draw [->] (A01) ..controls (72bp,393bp) and (72bp,383bp)  .. (A02);
  \draw [->] (A23) ..controls (118bp,193bp) and (112bp,182bp)  .. (A24);
  \draw [->] (A12) ..controls (99bp,293bp) and (99bp,283bp)  .. (A13);
  \draw [->] (A05) ..controls (26bp,193bp) and (32bp,182bp)  .. (A15);
  \draw [->] (A14) ..controls (80bp,193bp) and (86bp,182bp)  .. (A24);
  \draw [->] (A02) ..controls (64bp,343bp) and (58bp,332bp)  .. (A03);
  \draw [->] (A24) ..controls (83bp,142bp) and (70bp,130bp)  .. (A25);
  \draw [->] (A04) ..controls (53bp,243bp) and (59bp,232bp)  .. (A14);
  \draw [->] (A25) ..controls (53bp,93bp) and (59bp,82bp)  .. (A35);
  \draw [->] (A04) ..controls (37bp,243bp) and (31bp,232bp)  .. (A05);
  \draw [->] (A02) ..controls (80bp,343bp) and (86bp,332bp)  .. (A12);
  \draw [->] (A15) ..controls (45bp,143bp) and (45bp,133bp)  .. (A25);
  \draw [->] (A14) ..controls (64bp,193bp) and (58bp,182bp)  .. (A15);
  \draw [->] (A03) ..controls (61bp,292bp) and (74bp,280bp)  .. (A13);
  \draw [->] (A13) ..controls (91bp,243bp) and (85bp,232bp)  .. (A14);
\end{tikzpicture}
\qquad\qquad
\begin{tikzpicture}[>=latex',line join=bevel,rotate=180]
\node (A234) at (126bp,157bp) [draw,draw=none] {$\begin{tikzpicture}[scale=.2] \path[fill=gray!50] (0,0)  -- ++(1,0) -- ++(1,0)-- ++(0,1)-- ++(0,1)-- ++(0,1) -- ++(1,0) -- (0,3) -- cycle ;  \draw[thick] (0,0)  -- ++(1,0) -- ++(1,0)-- ++(0,1)-- ++(0,1)-- ++(0,1) -- ++(1,0); \draw[gray] (0,0) grid (3,3); \end{tikzpicture}$};
  \node (A123) at (126bp,307bp) [draw,draw=none] {$\begin{tikzpicture}[scale=.2] \path[fill=gray!50] (0,0)  -- ++(1,0)-- ++(0,1)-- ++(0,1)-- ++(0,1) -- ++(1,0) -- ++(1,0) -- (0,3) -- cycle ;  \draw[thick] (0,0)  -- ++(1,0)-- ++(0,1)-- ++(0,1)-- ++(0,1) -- ++(1,0) -- ++(1,0); \draw[gray] (0,0) grid (3,3); \end{tikzpicture}$};
  \node (A245) at (72bp,57bp) [draw,draw=none] {$\begin{tikzpicture}[scale=.2] \path[fill=gray!50] (0,0)  -- ++(1,0) -- ++(1,0)-- ++(0,1) -- ++(1,0)-- ++(0,1)-- ++(0,1) -- (0,3) -- cycle ;  \draw[thick] (0,0)  -- ++(1,0) -- ++(1,0)-- ++(0,1) -- ++(1,0)-- ++(0,1)-- ++(0,1); \draw[gray] (0,0) grid (3,3); \end{tikzpicture}$};
  \node (A045) at (18bp,157bp) [draw,draw=none] {$\begin{tikzpicture}[scale=.2] \path[fill=gray!50] (0,0) -- ++(0,1) -- ++(1,0) -- ++(1,0) -- ++(1,0)-- ++(0,1)-- ++(0,1) -- (0,3) -- cycle ;  \draw[thick] (0,0) -- ++(0,1) -- ++(1,0) -- ++(1,0) -- ++(1,0)-- ++(0,1)-- ++(0,1); \draw[gray] (0,0) grid (3,3); \end{tikzpicture}$};
  \node (A235) at (99bp,107bp) [draw,draw=none] {$\begin{tikzpicture}[scale=.2] \path[fill=gray!50] (0,0)  -- ++(1,0) -- ++(1,0)-- ++(0,1)-- ++(0,1) -- ++(1,0)-- ++(0,1) -- (0,3) -- cycle ;  \draw[thick] (0,0)  -- ++(1,0) -- ++(1,0)-- ++(0,1)-- ++(0,1) -- ++(1,0)-- ++(0,1); \draw[gray] (0,0) grid (3,3); \end{tikzpicture}$};
  \node (A014) at (45bp,357bp) [draw,draw=none] {$\begin{tikzpicture}[scale=.2] \path[fill=gray!50] (0,0) -- ++(0,1)-- ++(0,1) -- ++(1,0) -- ++(1,0)-- ++(0,1) -- ++(1,0) -- (0,3) -- cycle ;  \draw[thick] (0,0) -- ++(0,1)-- ++(0,1) -- ++(1,0) -- ++(1,0)-- ++(0,1) -- ++(1,0); \draw[gray] (0,0) grid (3,3); \end{tikzpicture}$};
  \node (A015) at (18bp,307bp) [draw,draw=none] {$\begin{tikzpicture}[scale=.2] \path[fill=gray!50] (0,0) -- ++(0,1)-- ++(0,1) -- ++(1,0) -- ++(1,0) -- ++(1,0)-- ++(0,1) -- (0,3) -- cycle ;  \draw[thick] (0,0) -- ++(0,1)-- ++(0,1) -- ++(1,0) -- ++(1,0) -- ++(1,0)-- ++(0,1); \draw[gray] (0,0) grid (3,3); \end{tikzpicture}$};
  \node (A012) at (72bp,457bp) [draw,draw=none] {$\begin{tikzpicture}[scale=.2] \path[fill=gray!50] (0,0) -- ++(0,1)-- ++(0,1)-- ++(0,1) -- ++(1,0) -- ++(1,0) -- ++(1,0) -- (0,3) -- cycle ;  \draw[thick] (0,0) -- ++(0,1)-- ++(0,1)-- ++(0,1) -- ++(1,0) -- ++(1,0) -- ++(1,0); \draw[gray] (0,0) grid (3,3); \end{tikzpicture}$};
  \node (A013) at (72bp,407bp) [draw,draw=none] {$\begin{tikzpicture}[scale=.2] \path[fill=gray!50] (0,0) -- ++(0,1)-- ++(0,1) -- ++(1,0)-- ++(0,1) -- ++(1,0) -- ++(1,0) -- (0,3) -- cycle ;  \draw[thick] (0,0) -- ++(0,1)-- ++(0,1) -- ++(1,0)-- ++(0,1) -- ++(1,0) -- ++(1,0); \draw[gray] (0,0) grid (3,3); \end{tikzpicture}$};
  \node (A025) at (18bp,257bp) [draw,draw=none] {$\begin{tikzpicture}[scale=.2] \path[fill=gray!50] (0,0) -- ++(0,1) -- ++(1,0)-- ++(0,1) -- ++(1,0) -- ++(1,0)-- ++(0,1) -- (0,3) -- cycle ;  \draw[thick] (0,0) -- ++(0,1) -- ++(1,0)-- ++(0,1) -- ++(1,0) -- ++(1,0)-- ++(0,1); \draw[gray] (0,0) grid (3,3); \end{tikzpicture}$};
  \node (A024) at (72bp,307bp) [draw,draw=none] {$\begin{tikzpicture}[scale=.2] \path[fill=gray!50] (0,0) -- ++(0,1) -- ++(1,0)-- ++(0,1) -- ++(1,0)-- ++(0,1) -- ++(1,0) -- (0,3) -- cycle ;  \draw[thick] (0,0) -- ++(0,1) -- ++(1,0)-- ++(0,1) -- ++(1,0)-- ++(0,1) -- ++(1,0); \draw[gray] (0,0) grid (3,3); \end{tikzpicture}$};
  \node (A034) at (72bp,257bp) [draw,draw=none] {$\begin{tikzpicture}[scale=.2] \path[fill=gray!50] (0,0) -- ++(0,1) -- ++(1,0) -- ++(1,0)-- ++(0,1)-- ++(0,1) -- ++(1,0) -- (0,3) -- cycle ;  \draw[thick] (0,0) -- ++(0,1) -- ++(1,0) -- ++(1,0)-- ++(0,1)-- ++(0,1) -- ++(1,0); \draw[gray] (0,0) grid (3,3); \end{tikzpicture}$};
  \node (A023) at (99bp,357bp) [draw,draw=none] {$\begin{tikzpicture}[scale=.2] \path[fill=gray!50] (0,0) -- ++(0,1) -- ++(1,0)-- ++(0,1)-- ++(0,1) -- ++(1,0) -- ++(1,0) -- (0,3) -- cycle ;  \draw[thick] (0,0) -- ++(0,1) -- ++(1,0)-- ++(0,1)-- ++(0,1) -- ++(1,0) -- ++(1,0); \draw[gray] (0,0) grid (3,3); \end{tikzpicture}$};
  \node (A145) at (45bp,107bp) [draw,draw=none] {$\begin{tikzpicture}[scale=.2] \path[fill=gray!50] (0,0)  -- ++(1,0)-- ++(0,1) -- ++(1,0) -- ++(1,0)-- ++(0,1)-- ++(0,1) -- (0,3) -- cycle ;  \draw[thick] (0,0)  -- ++(1,0)-- ++(0,1) -- ++(1,0) -- ++(1,0)-- ++(0,1)-- ++(0,1); \draw[gray] (0,0) grid (3,3); \end{tikzpicture}$};
  \node (A345) at (72bp,7bp) [draw,draw=none] {$\begin{tikzpicture}[scale=.2] \path[fill=gray!50] (0,0)  -- ++(1,0) -- ++(1,0) -- ++(1,0)-- ++(0,1)-- ++(0,1)-- ++(0,1) -- (0,3) -- cycle ;  \draw[thick] (0,0)  -- ++(1,0) -- ++(1,0) -- ++(1,0)-- ++(0,1)-- ++(0,1)-- ++(0,1); \draw[gray] (0,0) grid (3,3); \end{tikzpicture}$};
  \node (A035) at (18bp,207bp) [draw,draw=none] {$\begin{tikzpicture}[scale=.2] \path[fill=gray!50] (0,0) -- ++(0,1) -- ++(1,0) -- ++(1,0)-- ++(0,1) -- ++(1,0)-- ++(0,1) -- (0,3) -- cycle ;  \draw[thick] (0,0) -- ++(0,1) -- ++(1,0) -- ++(1,0)-- ++(0,1) -- ++(1,0)-- ++(0,1); \draw[gray] (0,0) grid (3,3); \end{tikzpicture}$};
  \node (A124) at (126bp,257bp) [draw,draw=none] {$\begin{tikzpicture}[scale=.2] \path[fill=gray!50] (0,0)  -- ++(1,0)-- ++(0,1)-- ++(0,1) -- ++(1,0)-- ++(0,1) -- ++(1,0) -- (0,3) -- cycle ;  \draw[thick] (0,0)  -- ++(1,0)-- ++(0,1)-- ++(0,1) -- ++(1,0)-- ++(0,1) -- ++(1,0); \draw[gray] (0,0) grid (3,3); \end{tikzpicture}$};
  \node (A125) at (72bp,207bp) [draw,draw=none] {$\begin{tikzpicture}[scale=.2] \path[fill=gray!50] (0,0)  -- ++(1,0)-- ++(0,1)-- ++(0,1) -- ++(1,0) -- ++(1,0)-- ++(0,1) -- (0,3) -- cycle ;  \draw[thick] (0,0)  -- ++(1,0)-- ++(0,1)-- ++(0,1) -- ++(1,0) -- ++(1,0)-- ++(0,1); \draw[gray] (0,0) grid (3,3); \end{tikzpicture}$};
  \node (A135) at (72bp,157bp) [draw,draw=none] {$\begin{tikzpicture}[scale=.2] \path[fill=gray!50] (0,0)  -- ++(1,0)-- ++(0,1) -- ++(1,0)-- ++(0,1) -- ++(1,0)-- ++(0,1) -- (0,3) -- cycle ;  \draw[thick] (0,0)  -- ++(1,0)-- ++(0,1) -- ++(1,0)-- ++(0,1) -- ++(1,0)-- ++(0,1); \draw[gray] (0,0) grid (3,3); \end{tikzpicture}$};
  \node (A134) at (126bp,207bp) [draw,draw=none] {$\begin{tikzpicture}[scale=.2] \path[fill=gray!50] (0,0)  -- ++(1,0)-- ++(0,1) -- ++(1,0)-- ++(0,1)-- ++(0,1) -- ++(1,0) -- (0,3) -- cycle ;  \draw[thick] (0,0)  -- ++(1,0)-- ++(0,1) -- ++(1,0)-- ++(0,1)-- ++(0,1) -- ++(1,0); \draw[gray] (0,0) grid (3,3); \end{tikzpicture}$};
  \draw [->] (A145) ..controls (53bp,93bp) and (59bp,82bp)  .. (A245);
  \draw [->] (A124) ..controls (126bp,243bp) and (126bp,233bp)  .. (A134);
  \draw [->] (A024) ..controls (88bp,292bp) and (101bp,280bp)  .. (A124);
  \draw [->] (A023) ..controls (107bp,343bp) and (113bp,332bp)  .. (A123);
  \draw [->] (A135) ..controls (64bp,143bp) and (58bp,132bp)  .. (A145);
  \draw [->] (A025) ..controls (34bp,242bp) and (47bp,230bp)  .. (A125);
  \draw [->] (A014) ..controls (37bp,343bp) and (31bp,332bp)  .. (A015);
  \draw [->] (A013) ..controls (64bp,393bp) and (58bp,382bp)  .. (A014);
  \draw [->] (A134) ..controls (110bp,192bp) and (97bp,180bp)  .. (A135);
  \draw [->] (A235) ..controls (91bp,93bp) and (85bp,82bp)  .. (A245);
  \draw [->] (A234) ..controls (118bp,143bp) and (112bp,132bp)  .. (A235);
  \draw [->] (A013) ..controls (80bp,393bp) and (86bp,382bp)  .. (A023);
  \draw [->] (A034) ..controls (88bp,242bp) and (101bp,230bp)  .. (A134);
  \draw [->] (A035) ..controls (34bp,192bp) and (47bp,180bp)  .. (A135);
  \draw [->] (A135) ..controls (80bp,143bp) and (86bp,132bp)  .. (A235);
  \draw [->] (A125) ..controls (72bp,193bp) and (72bp,183bp)  .. (A135);
  \draw [->] (A123) ..controls (126bp,293bp) and (126bp,283bp)  .. (A124);
  \draw [->] (A124) ..controls (110bp,242bp) and (97bp,230bp)  .. (A125);
  \draw [->] (A023) ..controls (91bp,343bp) and (85bp,332bp)  .. (A024);
  \draw [->] (A035) ..controls (18bp,193bp) and (18bp,183bp)  .. (A045);
  \draw [->] (A034) ..controls (56bp,242bp) and (43bp,230bp)  .. (A035);
  \draw [->] (A025) ..controls (18bp,243bp) and (18bp,233bp)  .. (A035);
  \draw [->] (A024) ..controls (56bp,292bp) and (43bp,280bp)  .. (A025);
  \draw [->] (A024) ..controls (72bp,293bp) and (72bp,283bp)  .. (A034);
  \draw [->] (A245) ..controls (72bp,43bp) and (72bp,33bp)  .. (A345);
  \draw [->] (A045) ..controls (26bp,143bp) and (32bp,132bp)  .. (A145);
  \draw [->] (A134) ..controls (126bp,193bp) and (126bp,183bp)  .. (A234);
  \draw [->] (A015) ..controls (18bp,293bp) and (18bp,283bp)  .. (A025);
  \draw [->] (A014) ..controls (53bp,343bp) and (59bp,332bp)  .. (A024);
  \draw [->] (A012) ..controls (72bp,443bp) and (72bp,433bp)  .. (A013);
\end{tikzpicture}
\ee
\caption{\label{graphfig}From left to right, the Hasse diagrams of the posets $L(1,5)$, $L(2,4)$ and $L(3,3)$.}
\end{figure}

\subsection{Action of $\phi^\pm(u)$}
The next task is to unravel the Jordan block structure of the Cartan generators $\phi^\pm(u)$ in the basis $\ZZ A$.
It is helpful to factor $\phi^\pm(u)$ (understood here, by a slight abuse of notation, purely as a map $\V_\ais\to \V_\ais$) as follows
\be \phi^\pm(u) = \phiin (u) \circ \phiout (u) \circ \phimid(u),\label{factoredphi}\ee
with
\bea \phimid(u)\on \VV B 
&:=& \VV B \frac{(q^{-1} - q a u)^{|B|} (q - q^{-1}au)^{|\comp B|}}{(1-a u)^n} H_{\left[\frac{u}{1-au}\right]}(z_1,\dots,z_n) \nn\\
  \phiin(u)\on \VV B &:=& \VV B\prod_{j\in B}  \frac{q^{- 1} - q a_j u}{q^{-1}-qau} 
                    = \VV B E_{\left[\frac{ - qu}{q^{-1} - qau}\right]}(z_B)\\
\phiout(u)\on \VV B &:=& \VV B\prod_{j\notin B} \frac{q - q^{-1} a_j u}{q-q^{-1}au} 
                =   \VV B E_{\left[\frac{ - q^{-1} u}{q - q^{-1}au}\right]}(z_{\comp B})\eea
where 
\be H_{[t]}(y_1,\dots,y_m) = \prod_{i=1}^m \frac{1}{1-ty_i} = \sum_{r\geq 0} t^r h_r(y_1,\dots,y_m)\ee
is the generating function of the complete symmetric functions $h_r=s_{(r)}$ and
\be E_{[t]}(y_1,\dots,y_m) = \prod_{i=1}^m (1+ty_i) = \sum_{r\geq 0} t^r e_r(y_1,\dots,y_m)\ee
is the generating function of the elementary symmetric functions $e_r = s_{(1^r)}$. 
It is clear that in the coincident limit only the (diagonal) leading term of $\phimid(u)$ survives; the rest, proportional to complete symmetric polynomials in $z_1,\dots,z_n$ of strictly positive degrees, go to zero. It is rather $\phiin$ and $\phiout$ that are responsible for the formation of the Jordan blocks, as follows. 

For all $A\subseteq \range 1 n$, let $\mathsf e(A,r)$ (resp. $\mathsf h(A,r)$) denote the set of all subsets $C \subseteq \range 1 n$ with $|A|$ elements such that the Ferrers diagram of $\lambdaof C$ is obtained from that of $\lambdaof A$ by adding $r$ boxes, no two in any one row (resp. column).
\begin{proposition} \label{phiaction}In the coincident limit,
\be
\phiin(u) \on \ZZ A =  
                       \sum_{r\geq 0} \left(\tfrac{ - qu}{q^{-1} - qau}\right)^r 
                        \sum_{C \in \mathsf e(A,r)} \ZZ C \, ;
\qquad \phiout(u) \on \ZZ A =  
                \sum_{r\geq 0} \left(\tfrac{ q^{-1}u}{q - q^{-1}au}\right)^r 
                        \sum_{C \in \mathsf h(A,r)} \ZZ C \, ;\nonumber \ee
and $\left[\phiin(u),\phiout(u)\right] \on\ZZ A= 0$.\end{proposition} 
\begin{proof}
Consider first the action of $\phiin(u)$ on the $\ZZ A$ basis vectors. By definition \ref{Zdef},
\bea \phiin(u) \on \ZZ A &=&  \sum_{\substack{B\subseteq \range 1 n\\ |B|=|A|}} \VV B s_{\lambdaof A}(z_B) E_{\left[\frac{ -  qu}{q^{-1} - qau}\right]}(z_B) \\
&=&   \sum_{r\geq 0} \left(\tfrac{ - qu}{q^{-1} - qau}\right)^r 
        \sum_{\substack{B\subseteq \range 1 n\\ |B|=|A|}} \VV B s_{\lambdaof A}(z_B) e_r(z_B). \nn\eea
Recall the Pieri formula $s_{\lambdaof A} e_r = \sum_{\lambda'}s_{\lambda'}$, where the sum is over every partition $\lambda'$ whose Ferrers diagram can be obtained from that of $\lambdaof A$ by adding $r$ extra boxes, no two in the same row  (see \eg \cite{MacDonald}). Since we are concerned with symmetric polynomials of a specific number, $|A|$, of variables, we can ignore those $\lambda'$ that have more than $|A|$ parts.
Moreover, if a given $\lambda'$ has any part greater than $n- |A|$ then the sequence $(\lambda'_{|A|+1-i}+i)_{\ir{|A|}}$ has an element greater than $n$, so that when $\sum_{\substack{B\subseteq \range 1 n\\ |B|=|A|}} \VV B s_{\lambda'}(z_B)$ is expressed as a sum of the $\ZZ C$, \textit{c.f.} (\ref{leftin}), the coefficients are elements of $\mathbb C[z_1,\dots,z_n]^{S_n}$ of strictly positive degree and therefore vanish in the coincident limit.
Thus, only partitions $\lambda'$ whose Ferrers diagrams ``fit'', in the obvious sense, inside the $(n-|A|)\times |A|$ ``frame'' defined by the lattice path of $A$, \textit{c.f.} figure \ref{staircasefig}, contribute in the limit. These are the partitions of the form $\lambda' = \lambdaof {A'}$ for some subset $A'\subseteq \range 1 n$ with $|A'| = |A|$.

Turning to $\phiout(u)$, we have from definition \ref{Zdef} and proposition \ref{VtoZ} that
\bea \phiout(u) \on \ZZ A &=&  \sum_{\substack{B\subseteq \range 1 n\\ |B|=|A|}}  \VV B  s_{\lambdaof A}(z_B)
                              E_{\left[\frac{ - q^{-1}u}{q - q^{-1}au}\right]}(z_{\comp B}) \\
&=&  \sum_C \sgn(w_{\comp C}) \ZZ C \sum_{r\geq 0} \left(\tfrac{ - q^{-1}u}{q - q^{-1}au}\right)^r \nn\\
&&\qquad\qquad\qquad\sum_B s_{\lambdaof A}(z_B) e_r(z_{\comp B}) s_{\lambdaof {\comp C}} (z_{\comp B}) \frac{1}{\prod_{\substack{i\in B\\j\in\comp B}} z_i-z_j}. \nn\eea
Here we may again use the Pieri formula, now for the polynomials in the $z_{\comp B}$, and, arguing as above, conclude that the non-vanishing summands in the coincident limit are those where ${\comp A} \in \mathsf e(\comp C,r)$; 
or, equivalently, $C \in \mathsf h(A,r)$. 
The result follows, on using (\ref{epssigneqn}).

Finally, to verify that $\phiin(u)$ and $\phiout(u)$ continue to commute in the coincident limit (so that the order of the factors in (\ref{factoredphi}) remains unimportant), it suffices to note the following: For any $r,s\geq 0$, the triplets of partitions $(\lambda,\lambda',\lambda'')$ such that $\lambda'$ can be obtained from $\lambda$ by adding $r$ boxes, no two in any one row, and $\lambda''$ can be obtained from $\lambda'$ by adding $s$ boxes, no two in any one column, are in bijection with the triplets $(\lambda,\lambda''',\lambda'')$ such that $\lambda'''$ can be obtained from $\lambda$ by adding $s$ boxes, no two in one column, and $\lambda''$ can be obtained from $\lambda'''$ by adding $r$ boxes, no two in one row. One way to see this is as follows. The \emph{skew diagram} $\theta=\lambda''-\lambda$ can be written as a disjoint union, $\theta=\sqcup_i\theta_i$, of finitely many \emph{connected components} $\theta_i$, each of which is a \emph{border strip} (see \eg 
 \cite{MacDonald} for the definitions). Consider each $\theta_i$ in turn. Label every box of $\theta_i$ by an arrow, \tikz[baseline=1pt,scale=.4]{\sqr{(0,0)}} if the box is in $\lambda'$, \tikz[baseline=1pt,scale=.4]{\sqd{(0,0)}} otherwise. By construction, precisely one of the following must hold: (i)  every row has exactly one \tikz[baseline=1pt,scale=.4]{\sqr{(0,0)}} or (ii) every column has exactly one \tikz[baseline=1pt,scale=.4]{\sqd{(0,0)}}. In case (i), reverse the sequence of arrows in each row; in case (ii), reverse the sequence of arrows in each column. Repeat for each $\theta_i$. The resulting labelling defines $\lambda'''$. This map is clearly involutive. As examples,
\be \tikz[baseline=30,scale=.4]
{\sqr{(1,1)} \sqr{(1,2)} \sqd{(2,2)} \sqd{(3,2)} \sqr{(3,3)} \sqr{(3,4)} \sqd{(4,4)} \sqd{(5,4)} \sqd{(6,4)}} 
\quad\leftrightarrow\quad \tikz[baseline=30,scale=.4]
{\sqr{(1,1)} \sqd{(1,2)} \sqd{(2,2)} \sqr{(3,2)} \sqr{(3,3)} \sqd{(3,4)} \sqd{(4,4)} \sqd{(5,4)} \sqr{(6,4)}}\ee
are a pair of border strips of type (i), while
\be \tikz[baseline=30,scale=.4]
{\sqd{(1,1)} \sqr{(1,2)} \sqd{(2,2)} \sqd{(3,2)} \sqr{(3,3)} \sqr{(3,4)} \sqd{(4,4)} \sqd{(5,4)} \sqd{(6,4)}}
\quad\leftrightarrow\quad \tikz[baseline=30,scale=.4]
{\sqr{(1,1)} \sqd{(1,2)} \sqd{(2,2)} \sqr{(3,2)} \sqr{(3,3)} \sqd{(3,4)} \sqd{(4,4)} \sqd{(5,4)} \sqd{(6,4)}} \ee
are a pair of type (ii).
\finproof 
\end{proof}

\begin{remark}
Note that $\frac{ - q^{\pm 1}u}{q^{\mp 1} - q^{\pm 1}au}$ has expansion $0-q^{\pm 2}u +O(u^2)$ about $u=0$, but expansion $1/a + O(u^{-1})$ about $u=\8$. This originates in the choice of $z_j=a_j-a$ rather than $a_j^{-1}-a^{-1}$ in definition \ref{Zdef}. A consequence is that $\phi^+_0$ is manifestly diagonal, while $\phi^-_0$ is not. However, it follows from the definition (\ref{phiu}) that  $\phi^+_0\phi^-_0=\id$ 
and hence that $\phi^-_0$ must be diagonal too. It is interesting to check this explicitly. 
Consider using proposition \ref{phiaction} to expand $\phi_\wh(u)|_{u=\8} \circ \phi_\bl(u)|_{u=\8} \on \ZZ A$ as a power series in $a^{-1}$. For any $k>0$ the coefficient of the $k$th power is proportional to $\sum_D \sum_{j=0}^k (-1)^j  \sum_C \ZZ D$,
where $\sum_D$ is over $D$ such that $\lambdaof D-\lambdaof A$ has $k$ boxes, 
while $\sum_C$ is over all $C$, if any, such that $C \in \mathsf e(A,k-j)$ and $D \in \mathsf h(C, j)$. 
Labelling the boxes  of $\lambdaof D-\lambdaof A$ by \tikz[baseline=1pt,scale=.3]{\sqr{(0,0)}} or  \tikz[baseline=1pt,scale=.3]{\sqd{(0,0)}} as in the proof above, there is, for any given $D$, 
 an involution $(j,C)\leftrightarrow (j',C')$, $|j'-j|= 1$, on the allowed pairs $(j,C)$, defined by flipping the arrow of the top-right box of $\lambdaof D-\lambdaof A$. For example, the following border strips are paired,
\be \tikz[baseline=30,scale=.4]
{\sqr{(1,1)} \sqd{(1,2)} \sqd{(2,2)} \sqr{(3,2)} \sqr{(3,3)} \sqd{(3,4)} \sqd{(4,4)} \sqd{(5,4)} \sqr{(6,4)}}
\quad\leftrightarrow\quad  \tikz[baseline=30,scale=.4]
{\sqr{(1,1)} \sqd{(1,2)} \sqd{(2,2)} \sqr{(3,2)} \sqr{(3,3)} \sqd{(3,4)} \sqd{(4,4)} \sqd{(5,4)} \sqd{(6,4)}}\ee
The alternating factor $(-1)^j$ ensures that the paired terms cancel in the sum.
\end{remark}

By similar arguments it is also possible to compute the explicit actions of the raising and lowering operators in the $\ZZ A$ basis (which gives a more direct demonstration that the basis is well-behaved). For brevity we shall simply state without proof the following
\begin{proposition}
Let \be \mathscr C^\mp(r,p,t,m) = a^{t+m-r-p} \sum_{s=0}^p (-1)^{p-s} q^{\pm(r-p+s)} (q^{\pm 1}-q^{\mp 1})^{m+s-r-p} \binom t s \binom {m-r} {p-s}.\ee
Then, in the coincident limit, for all $t\in\mathbb Z$ and all $A\subseteq \range 1 n$:
\be x^-_t\on\ZZ A =  \sum_{p=0}^{n-1}  \left(-1\right)^{|\{j\in A: j<p+1\}|} \sum_{r=0}^{|A|} \mathscr C^-(r,p,t,|A|) \sum_{C:C\setminus\{p+1\} \in \mathsf e(A,r)} Z_C \, ;\ee
and
\be x^+_t\on\ZZ A = (-1)^{|\comp A|} \sum_{p=0}^{n-1}  \left(-1\right)^{|\{j\in A: j>n-p\}|} \sum_{r=0}^{|\comp A|} (-1)^r \mathscr C^+(r,p,t,|\comp A|) \sum_{C:C\cup\{n-p\} \in \mathsf h(A,r)} Z_C \, .\ee
\end{proposition}

\subsection{Jordan block structure}\label{Sec:JB} 
\begin{definition}
For each $m\in\range 0 n$ let 
\be\Vmn = \mathrm{span}_{\mathbb C} \left\{\ZZ A:|A|=m\right\}\ee 
and, for all $k\geq 0$, let
\be F_k \Vmn = \mathrm{span}_{\mathbb C} \left\{\ZZ A: |A|=m, \, |\lambdaof A| \geq k\right\},\ee
\be G_k \Vmn = \mathrm{span}_{\mathbb C} \left\{\ZZ A: |A|=m, \, |\lambdaof A|  = k\right\},\ee
and
\be \pi_k : \Vmn \to G_k\Vmn;\quad \sum_{\substack{A\subseteq \range 1 n\\ |A| = m}} v_A \ZZ A \mapsto 
                                   \sum_{\substack{A\subseteq \range 1 n\\ |A| = m\\ |\lambdaof A|=k }} v_A \ZZ A.\ee  
\end{definition}
The $F_k\Vmn$ define a decreasing filtration of $\Vmn$,
\be \Vmn = F_{0} \Vmn \supsetneq F_1 \Vmn \supsetneq\dots\supsetneq F_{m(n-m)} \Vmn\supsetneq F_{m(n-m)+1} \Vmn = \{0\},\ee
with $G_k \Vmn \cong {F_k\Vmn}\big/{F_{k+1}\Vmn}\, \forall k\geq 0$ the associated gradation: 
\be \Vmn = \bigoplus_{k=0}^{m(n-m)} G_k \Vmn,\qquad v = \sum_{k=0}^{m(n-m)} \pi_k(v)  \quad \forall v\in \Vmn\,.\ee
Now denote by
\be \label{t-binomial} {{n}\choose{m}}_{\!\! t} = \frac{(n)_t}{(m)_t(n-m)_t}\, ,  \qquad \mbox{with} \qquad (n)_t = \prod_{k=1}^{n} \frac{1-t^{k}}{1-t}\, ,\ee
the $t$-binomial coefficients.
It is known -- see \eg \cite{MacDonald} -- that these yield the counting function for the dimensions of the subspaces of definite grade:
\be \binomt n m = \sum_{k=0}^{m(n-m)} t^k \dim G_k \Vmn.\label{tbcf}\ee
\begin{example} In the case $n=4$, $m=2$, $\dim G_0=\dim G_1 = \dim G_3 = \dim G_4=1$ and $\dim G_2 = 2$, which can be pictured as follows:
\be \begin{matrix}
\binomt 4 2 =& 1 &+& t &+& 2t^2 &+& t^3 &+& t^4 \\ &
\tikz[baseline=0pt,scale=.5]{
\path [fill=gray!50] (0,0) -- (1,0) -- (2,0) -- (2,1) -- (2,2) -- (2,0)--cycle;
\draw [very thick]   (0,0) -- (0,1) -- (0,2) -- (1,2) -- (2,2);
\draw[gray] (0,0) grid (2,2); } &&
\tikz[baseline=0pt,scale=.5]{
\path [fill=gray!50] (0,1) -- (1,1) -- (1,2) -- (0,2) -- cycle;
\draw [very thick]   (0,0) -- (0,1) -- (1,1) -- (1,2) -- (2,2);
\draw[gray] (0,0) grid (2,2); } &&
\tikz[baseline=0pt,scale=.5]{
\path [fill=gray!50] (0,1) -- (2,1) -- (2,2) -- (0,2) --cycle;
\draw [very thick]   (0,0) -- (0,1) -- (1,1) -- (2,1) -- (2,2);
\draw[gray] (0,0) grid (2,2); } &&
\tikz[baseline=0pt,scale=.5]{
\path [fill=gray!50] (0,0) -- (1,0) -- (1,1) -- (2,1) -- (2,2) -- (0,2)--cycle;
\draw [very thick]   (0,0) -- (1,0) -- (1,1) -- (2,1) -- (2,2);
\draw[gray] (0,0) grid (2,2); } &&
\tikz[baseline=0pt,scale=.5]{
\path [fill=gray!50] (0,0) -- (0,1) -- (0,2) -- (1,2) -- (2,2) -- (2,0)--cycle;
\draw [very thick]   (0,0) -- (1,0) -- (2,0) -- (2,1) -- (2,2);
\draw[gray] (0,0) grid (2,2); } \\
& && &&
\tikz[baseline=0pt,scale=.5]{
\path [fill=gray!50] (0,0) -- (1,0) -- (1,1) -- (1,2) -- (0,2) --cycle;
\draw [very thick]   (0,0) -- (1,0) -- (1,1) -- (1,2) -- (2,2);
\draw[gray] (0,0) grid (2,2); } && &&
\end{matrix}\ee
\end{example}

In view of (\ref{factoredphi}) and proposition \ref{phiaction}, each $\Vmn$ is an $l$-weight space -- \textit{c.f.} (\ref{lweightspacesdef}) -- of the module $\V_{(a^n)}$: let
\be \psimnu = \sum_{r=0}^\8 \psimnur u^{\pm r}:= \phi^\pm(u) - \gamma_{m,n}^\pm(u)\,\id,\quad \gamma_{m,n}(u) = \frac{(q^{-1} - q a u)^m (q - q^{-1}au)^{n-m}}{(1-a u)^n};\nn \ee
then it is clear from proposition \ref{phiaction} and the remark following that $\psi^\pm_{m,n; 0}=0$ and that $\psimnu$ is strictly upper triangular with respect to the filtration $\{F_k\Vmn\}$, in the sense that
\be \psimnu\left( F_k\Vmn \right) \subseteq  F_{k+1} \Vmn[[u^{\pm 1}]] \ee
for all $k\in\range 0 {m(n-m)-1}$. In order to compute the lengths and multiplicities of the Jordan chains of $\phi^\pm_{\pm r}$, $r\in \mathbb N_0$, it will actually suffice to consider only the leading superdiagonal elements of $\psimnu$, \ie to consider the map $\sum_{k=0}^{m(n-m)} \pi_{k+1} \circ \psimnu \circ \pi_k$. Define $N_{m,n}^{\pm}(u) = \sum_{r=1}^{\8} N^\pm_{m,n;\pm r} u^{\pm r} \in \mathbb C[[u^{\pm 1}]]$ by setting, for all $r \in \mathbb N^*$,
\be \label{Nmnu} N_{m,n;\pm r}^\pm :=  a^{\pm r-1} q^{\pm n - (1 \pm 1) m} (q^{-2} - q^2) P_{m,n;r-1}(q) \ee
with
\be P_{m,n;r}(q):=\sum_{p_1=r-n+2}^r (-1)^{r-p_1} {{n+p_1-1}\choose {p_1}} \sum_{p_2=0}^{m-1} {{m-1}\choose {p_2}}{{n-m-1}\choose {r-p_1-p_2}}q^{2(2p_2-r+p_1)}\,.\nn\ee
Define also the linear map \be X : \Vmn \to \Vmn;\quad X\on \ZZ A = \sum_{C \in \mathsf e(A,1)} \ZZ C \label{Xdef}\, .\ee
\begin{proposition}
\label{prop:psiNmnu}
For all $k \in \range{0}{m(n-m)}$,
\be \pi_{k+1} \circ \psimnu \circ \pi_k = N_{m,n}^\pm(u) \, X \, . \ee
\end{proposition}

\begin{proof}
By proposition \ref{phiaction}, we have
\be \pi_{k+1} \circ \psimnu \circ \pi_k =  \gamma_{m,n}(u) \left(\frac{- qu}{q^{-1} - qau} +\frac{q^{-1}u}{q - q^{-1}au}\right) X \, . \ee
The result follows by expanding the right hand side into formal power series in $u^{\pm 1}$. \finproof
\end{proof}

We need to understand the Jordan chains of $X$. This problem has been solved in \cite{Proctor} (where our $\Vmn$ is $\widetilde L(m,n-m)$). The set $L(m,n-m) := \{A\subseteq \range 1 n: |A|=m\}$ has a partial ordering in which $C\geq A$ if and only if the Ferrers diagram of $\lambda^C$ covers that of $\lambda^A$. $L(m,n-m)$ is moreover a ranked poset, and the ranking is identical with the grading $G_k$ in an obvious sense. Figure \ref{graphfig} shows, for certain examples, the Hasse diagram of this poset, \ie the directed acyclic graph whose edges are those pairs $(A\to C)$ such that the Ferrers diagram of $\lambdaof C$ is obtained from that of $\lambdaof A$ by adding one box. 
 Following \cite{Proctor}, let us further define
\bea H : \Vmn \to \Vmn;\quad H\on \ZZ A &=& \ZZ A \left(2|\lambdaof A| - m(n-m)\right) \nn\\
     Y : \Vmn \to \Vmn;\quad Y\on \ZZ A &=& \sum_{B: A \in \mathsf e(B,1)} \ZZ B (n-b)b \label{HYdef}\eea
where $b$ is the unique element of $B\setminus A$. It is then a straightforward exercise to verify 
\begin{proposition}[\cite{Proctor}]\label{proctorprop}
These maps $X$, $Y$, $H$ are a realization of $\mf{sl}_2$, \ie
\be \left[H,X\right] = 2X \quad \left[H,Y\right] = -2Y,\quad \left[X,Y\right] = H.\ee 
\end{proposition}
Under this action, $\Vmn$ decomposes into a direct sum of irreducible representations, $\Vmn=\oplus_{t=1}^T \Vmn^{(t)}$, $T\in \mathbb N$, whose lowest weight vectors we shall write as $(s_t)_{1\leq t\leq T}$. We have  
\be Y\on s_t = 0,\quad H\on s_t =  -2h_ts_t,\,\, h_t\in \alf\mathbb N\ee
and hence $s_t\in G_{\half m(n-m) -h_t}\Vmn$ and $\dim(\Vmn^{(t)})=2h_t+1$; let us choose to order the summands such that these dimensions form a weakly decreasing sequence. Then 
\be \left( X^p\on s_t\right)_{1\leq t\leq T; 0\leq p \leq {2h_t}} \ee
is a basis of $\Vmn$ in which by construction $X$ is in Jordan normal form. The Jordan chains are the irreps $\Vmn^{(t)}$ and we know, by considering how these irreps are constructed using the usual techniques of $\mf{sl}_2$ representation theory, that $\dim(\Vmn^{(t)})$ is equal to the number of distinct $k$ such that $\dim G_k\Vmn \geq t$. 
(Hence, in particular, the sequence $(\dim G_k\Vmn)_{1\leq k \leq m(n-m)}$ is unimodal and symmetric in $k\leftrightarrow m(n-m)-k$.)
For example, in the cases shown in figure \ref{graphfig}, we have Jordan chains of lengths
\begin{align} \V_{1,6}&: \qquad (6) \\
              \V_{2,6}&: \qquad (9,5,1) \\
              \V_{3,6}&: \qquad (10,6,4) .\end{align}
Next we note the following elementary result. 
\begin{lemma}
\label{Flemma}
Suppose that $V=F_{k_0}V \supsetneq F_{k_0+1} \supsetneq \dots \supsetneq F_{k_1}V \supsetneq F_{k_1+1}V=\{0\}$ is a decreasing filtration of the finite dimensional $\mathbb C$-vector space $V$ 
and that, for some index sets $(I_k)_{k_0\leq k\leq k_1}$, $(v_{i,k})_{i\in I_k,k_0\leq k\leq k_1}$ is a basis of $V$ such that $v_{i,k}\in F_kV$ for all $i\in I_k$, $k_0\leq k\leq k_1$. If there exists a set $(w_{i,k})_{i\in I_k,k_0\leq k\leq k_1}$ of vectors in $V$ such that $w_{i,k} = v_{i,k} \mod F_{k+1}V$  for all $i\in I_k$, $k_0\leq k\leq k_1$ then  $(w_{i,k})_{i\in I_k,k_0\leq k\leq k_1}$ is also a basis of $V$.
\end{lemma}
The relevance of the above $\mathfrak{sl}_2$-decomposition to the Jordan structure of $\phi^\pm(u)$ is that
\begin{proposition} The dimensions and multiplicities of the Jordan blocks of $\phi^\pm_{\pm 1}|_{\Vmn}$ coincide with those of $X$. If in addition the deformation parameter $q \in \Cx$ is transcendental, the dimensions and multiplicities of the Jordan blocks of $\phi^\pm_{\pm r}|_{\Vmn}$ coincide with those of $X$, for all $r \in \mathbb N^*$.
\end{proposition}
\begin{proof}
We want to prove that there exists a set of vectors $(\tilde{s}_t)_{1 \leq t \leq T}$ in $\Vmn$ such that $\big((\psimnur)^p \on \tilde{s}_t\big)_{1 \leq t \leq T; 0 \leq p \leq 2h_t}$ is a basis of $\Vmn$ in which $\psimnur$ is in Jordan normal form. In order to do so, we first observe that, proposition \ref{prop:psiNmnu}, 
\be\label{psiXmod} (\psimnur)^p\on s_t= (N_{m,n;\pm r}^\pm)^p \, X^p\on s_t \mod F_{\half m(n-m) -h_t +p+1}\Vmn\ee
for all $1\leq t\leq T$, $0\leq p \leq {2h_t}$. As $q$ is not a root of unity, eq. (\ref{Nmnu}) implies that $N_{m,n;\pm 1}^\pm \neq 0$. If, furthermore, $q$ is transcendental, it cannot be a root of $P_{m,n;r}(q) \in \mathbb Q[q^{\pm 1}]$ and hence, eq. (\ref{Nmnu}), $N_{m,n;\pm r}^\pm \neq 0$, for all $r \in \mathbb N^*$. Since $\left( X^p\on s_t\right)_{1\leq t\leq T; 0\leq p \leq {2h_t}}$ is a basis of $\Vmn$, lemma \ref{Flemma} implies that $\left( \psimnur^p\on \tilde{s}_t\right)_{1\leq t\leq T; 0\leq p \leq {2h_t}}$ is also a basis of $\Vmn$ for any set $(\tilde{s}_t)_{1 \leq t \leq T}$ of vectors such that
\be\tilde{s}_t = s_t \mod F_{\half m(n-m) - h_t +1} \Vmn \label{stildes}\ee
for all $t \in\range{1}{T}$. 
Thus, it suffices to prove that a set $(\tilde{s}_t)_{1 \leq t \leq T}$ satisfying (\ref{stildes}) can be chosen in such a way that,  for every $t \in \range{1}{T}$
\be (\psimnur)^{2h_t+1}\on\tilde{s}_t = 0\, .\ee
This is obviously true for $t=1$, as $h_1=m(n-m)/2$ implies that
\be (\psimnur)^{2h_1+1}\on s_{1}= (N_{m,n;\pm r}^\pm)^{2h_1+1} \, X^{2h_1+1}\on s_1 = 0\, ,\ee
and $((\psimnur)^p\on s_1)_{0 \leq p\leq 2h_1}$ is thus a Jordan chain of $\psimnu$. We can therefore set $\tilde{s}_1 := s_1$ which trivially satisfes (\ref{stildes}). 
Suppose then that for some $1\leq t' < T$ we have found $(\tilde{s}_t)_{1 \leq t \leq t'}$ such that $\left( (\psimnur)^p\on \tilde{s}_t\right)_{1 \leq t \leq t'; 0\leq p \leq {2h_t}}$ are Jordan chains of $\psimnur$ and (\ref{stildes}) holds for every $t \in \range{1}{t'}$. Consider the $(t'+1)$st chain. By (\ref{psiXmod}), we have 
\be(\psimnur)^{2h_{t'+1}+1}\on s_{t'+1} = 0 \mod F_{\half m (n-m) + h_{t'+1}+2} \Vmn\,.\ee 
Now note that $\left( X^p\on s_t\right)_{1\leq t\leq t'; p \geq {h_{t'+1}+h_t+2}}$ is a basis of $F_{\half m(n-m) +h_{t'+1} +2} \Vmn$ and that therefore, in view of (\ref{psiXmod}),
\be \left( (\psimnur)^p\on \tilde{s}_t\right)_{1\leq t\leq t'; p \geq h_{t'+1}+h_t +2} \ee
is also a basis of $F_{\half m(n-m) +h_{t'+1} +2}\Vmn$ by lemma \ref{Flemma} above. Hence, there exists a set $(\alpha_{t,p})_{1\leq t\leq t'; p \geq h_{t'+1}+h_t+2}$ of elements of $\mathbb C$ such that
\be (\psimnur)^{2h_{t'+1}+1}\on s_{t'+1} = \sum_{\substack{1\leq t\leq t'\\ p \geq h_{t'+1}+h_t+2}} \alpha_{t,p} \, (\psimnur)^p \on \tilde{s}_t \, .\ee
Thus, setting 
\be \tilde{s}_{t'+1} := s_{t'+1} - \sum_{\substack{1\leq t\leq t'\\ p \geq h_{t'+1}+h_t+2}} \alpha_{t,p}\, \, (\psimnur)^{p-2h_{t'+1}-1} \on \tilde{s}_t \ee 
makes $((\psimnur)^p \on \tilde{s}_{t'+1})_{0 \leq p \leq h_{t'+1}}$ a Jordan chain of $\psimnur$. Moreover,  (\ref{stildes}) now clearly holds for $t'+1$. \finproof
\end{proof}

\section{The general standard module}\label{Sec:multiplepoints}
Having treated the standard module $\V_\ais$ in the limit in which all the elements of $\ais$ converge to a common value, it remains to consider the general case. Thus, let $(a^{(s)})_{s\in\range 1 r}$ be an $r$-tuple of pairwise distinct elements of $\Cx$ and $(n_s)_{s\in\range 1 r}$ a composition of $n$. We consider the standard module $\V_{(a^{(1)}_1,\dots,a^{(1)}_{n_1}, a^{(2)}_{1},\dots, a^{(2)}_{n_2}, \dots,\dots, a^{(r)}_1,\dots,a^{(r)}_{n_r})}$ where, for each $s\in \range 1 r$,
\be a^{(s)}_i = a^{(s)} + z^{(s)}_i, \qquad z_i^{(s)} =  \eps \alpha^{(s)}_i \qquad \forall i \in \range{1}{n_{s}} \label{multilimit}\ee
for pairwise distinct complex numbers $(\alpha^{(s)}_i)_{ i \in \range 1 {n_s}}$. We want to take the limit $\eps\to 0$. When $\eps\neq 0$ but sufficiently small, the $\VV A$ of (\ref{VAphiaction}-\ref{VAxmaction}) constitute a good basis. Let us write these basis vectors as $\VV{A^{(1)},\dots, A^{(r)}}$, with $A^{(s)} \subseteq \range 1 {n_s}$ for each $s\in \range 1 r$.
Then the change of basis in Definition \ref{Zdef} generalizes to
\begin{definition}\label{Zdefgen}
For all $A^{(s)}\subseteq \range 1 {n_s}$, $s\in\range 1 r$, let
\be \ZZ{A^{(1)},\dots, A^{(r)}}  = \sum_{\substack{B^{(1)}\subseteq \range 1{n_1} \\  |B^{(1)}|=|A^{(1)}|}} \dots 
  \sum_{\substack{B^{(r)}\subseteq \range 1{n_r} \\  |B^{(r)}|=|A^{(r)}|}}  
\VV{B^{(1)},\dots,B^{(r)}}  \prod_{s=1}^r s_{\lambdaof {A^{(s)}}}(z^{(s)}_{B^{(s)}}).\nn\ee
\end{definition}
The natural generalizations of the inverse change of basis (Proposition \ref{VtoZ}) and Proposition \ref{innerprop}  have entirely analogous proofs. 
Hence, by Proposition \ref{loweringprop}, one finds that $x^-_{\tau_1}\on \dots \on x^-_{\tau_m} \on \VV\emptyset$ is a linear combination of  the $\ZZ{A^{(1)},\dots,A^{(r)}}$ with coefficients (in $\bigotimes_{s=1}^r \mathbb C[z^{(s)}_1,\dots z^{(s)}_{n_s}]^{S_{n_s}}$) that are regular in the limit $\eps\to 0$. So we have that the vectors $\ZZ{A^{(1)},\dots,A^{(r)}}$  of definition \ref{Zdefgen} constitute a well-defined basis of $\V_\ais$ in the limit defined by taking $\eps\to 0$ in (\ref{multilimit}).

On factoring $\phi^\pm(u)$ exactly as in (\ref{factoredphi}), the counterpart of proposition (\ref{phiaction}) is
\begin{proposition} \label{phiactiongen}In the limit defined by taking $\eps\to 0$ in (\ref{multilimit}),  $\phiin(u)$ and $\phiout(u)$ commute and
\bea
\phiin(u) \on \ZZ{A^{(1)},\dots,A^{(r)}} &=& 
   \sum_{t_1\geq 0}\dots\sum_{t_r\geq 0}  \sum_{C^{(1)} \in \mathsf e(A^{(1)},t_1)} \dots \sum_{C^{(r)} \in \mathsf e(A^{(r)},t_r)} \ZZ{C^{(1)},\dots,C^{(r)}}\prod_{s=1}^r  \left(\tfrac{-  qu}{q^{- 1} - q a^{(s)}u}\right)^{t_s}\nn\\
\phiout(u) \on \ZZ{A^{(1)},\dots,A^{(r)}} &=& 
   \sum_{t_1\geq 0}\dots\sum_{t_r\geq 0}  \sum_{C^{(1)}\in \mathsf h(A^{(1)},t_1)}\dots \sum_{C^{(r)}\in \mathsf h(A^{(r)},t_r)} \ZZ{C^{(1)},\dots,C^{(r)}}\prod_{s=1}^r \left(\tfrac{ q^{- 1}u}{q - q^{-1} a^{(s)}u}\right)^{t_s} \nn\eea
\end{proposition} 
\begin{definition}
For each $r$-tuple $(m_s)_{s\in\range 1 r}\in \mathbb N_0^r$ such that $m_s\leq n_s\,\forall s\in\range 1 r$ let 
\be\nn\Vmmmn := \mathrm{span}_{\mathbb C} \left\{ \ZZ{A^{(1)},\dots,A^{(r)}} : |A^{(s)}|=m_s\, \forall s\in \range 1 r \right\}.\ee
For all $k\geq 0$ let $G_k \Vmmmn$ be the gradation 
\be\nn G_k \Vmmmn := \Vmmmn \cap \mathrm{span}_{\mathbb C} \left\{ \ZZ{A^{(1)},\dots,A^{(r)}} : \sum_{s=1}^r |\lambdaof {A^{(s)}}| = k\right\},\ee
let $F_k\Vmmmn =\bigoplus_{\ell = k}^\8 G_k\Vmmmn$ be the corresponding filtration, and let $\pi_k : \Vmmmn \to G_k\Vmmmn$ be the linear map such that
\be\nn \pi_k:  \ZZ{A^{(1)},\dots,A^{(r)}} \mapsto 
\begin{cases}   \ZZ{A^{(1)},\dots,A^{(r)}} &\text{if}\,\, \sum_{s=1}^r |\lambdaof {A^{(s)}}|=k \\
                0 & \text{otherwise.} \end{cases}\ee
\end{definition}
The $\Vmmmn$ are the $l$-weight spaces, and indeed $\psimnu\left( F_k\Vmmmn \right) \subseteq  F_{k+1} \Vmmmn[[u^{\pm 1}]]$, where 
\be\label{gammamndef} \psimmmnu := \phi^\pm(u) - \gamma_{(m),(n)}^{\pm}(u) \id, \quad \gamma_{(m),(n)}(u) := \prod_{s=1}^r \frac{(q^{-1} - q a^{(s)} u)^{m_s} (q - q^{-1}a^{(s)}u)^{n_s-m_s}}{(1-a^{(s)} u)^{n_s}}.\nn\ee

Now define linear maps $\Delta^rX=\sum_{s=1}^r X_{(s)}, \Delta^rY=\sum_{s=1}^rY_{(s)}$ and $\Delta^rH=\sum_{s=1}^r H_{(s)}$ where, for each $s\in\range 1 r$, $X_{(s)},Y_{(s)},H_{(s)}$ are the maps of (\ref{Xdef}-\ref{HYdef}) acting in slot $s$.
It follows from proposition \ref{proctorprop} together with the fact that $\Delta^r$ is indeed the usual $r$-fold algebra-like coproduct of $\mf{sl}_2$ that
the maps $\Delta^rX$, $\Delta^rY$, $\Delta^rH$ are a realization of $\mf{sl}_2$. 
Consequently
\be \Vmmmn \cong \bigotimes_{s=1}^r \V_{m_s,n_s} \ee
as representations of $\mf{sl}_2$. In particular, the Jordan chains of $X$ are the $\mf{sl}_2$-irreducible components of this tensor product.

For definiteness, let us consider the mode $\psi^+_{(m),(n),1}$. The argument is similar for other modes but, by virtue of the strict upper-triangularity of $\psi^{\pm}_{(m),(n),\pm r}$ for all $r\in \mathbb N_0$, it is enough to consider one mode and to show that its Jordan chains are of maximal length consistent with the gradation $G_k$. 
The leading superdiagonal component of $\psi^+_{(m),(n),1}$ is proportional to $\Delta^rX$:
\be\sum_k \pi_{k+1} \circ \psi^+_{(m),(n),1} \circ \pi_k = \gamma_{(m),(n),0}^+ (q^{-2}-q^2) \Delta^r X \,  .\ee 
Therefore, by the same logic as in section \ref{Sec:JB}, the Jordan blocks of $\phi^+_1$ in $\Vmmmn$ have the same dimensions and multiplicities as those of $\Delta^rX$. 
Finally, we note that 
\be \dim G_k \Vmmmn = \sum_{k_1+\dots+k_r=k} \prod_{s=1}^r \dim G_{k_s} \V_{m_s,n_s},\ee
where the outer sum is over all $r$-compositions of $k$, so that the counting functions are multiplicative and hence, given (\ref{tbcf}),
\be \sum_{k=0}^{\sum_{s=1}^rm_s(n_s-m_s)} t^k \dim G_k \Vmmmn = \prod_{s=1}^r \sum_{k=0}^{m_s(n_s-m_s)} t^k \dim G_k\V_{m_s,n_s} 
                                                              = \prod_{s=1}^r \binomt{n_s}{m_s}.\label{Gkdims}\ee

We have therefore established the following theorem, which gives the dimensions of the 
Jordan grades 
for a general standard module $M(P)$.

\begin{theorem}
\label{THX}
Suppose $P(u)=\prod_{s=1}^r(1-ua^{(s)})^{n_s}$, with the $a^{(s)}\in \Cx$, $1\leq s\leq r$, pairwise distinct. Let $\gamma^\pm(u)=\sum_{r\in \mathbb N_0}\gamma^\pm_{(m),(n),\pm r}u^{\pm r}$ be the $l$-weight corresponding to the monomial $\prod_{s=1}^r  \binom{n_s}{m_s} Y_{a^{(s)}}^{n_s-m_s} Y_{a^{(s)}q^2}^{-m_s}$ in $\chi_q(M(P))$ and let, for all $k \in \mathbb N$,
\be \mathscr F_k M_\gamma(P)
:= \bigcap_{\substack{(r_1,\dots,r_{k}) \in\mathbb N_0^{k}\\ (\sigma_1,\dots,\sigma_{k})\in \{\pm\}^{k}}} \ker \prod_{\ell=1}^{k}\left(  \phi^{\sigma_\ell}_{\sigma_\ell r_\ell} - \gamma_{(m),(n);\sigma_\ell r_\ell}^{\sigma_\ell} \id \right)\nn\ee
denote the Jordan filtration of the corresponding $l$-weight space $M_\gamma(P)$. Then \be\nn\mathscr F_k M_\gamma(P) = M_{\gamma}(P)\cap \ker \left( \phi^+_{1} - \gamma_{(m),(n); 1}^{+} \id\right)^{k}\ee
and 
the sequence 
\be\left( \dim\left(\mathscr F_{k+1}M_\gamma(P)\big/\mathscr F_{k}M_\gamma(P)
\right)\right)_{0\leq k\leq \8}\nn\ee 
is given by sorting the coefficients of powers of $t$ in $\prod_{s=1}^r \binomt{n_s}{m_s}$ into weakly decreasing order. \end{theorem}


\section{{\bf \textit{q,t}}-characters}
\label{Sec:qtchar}
Thus far our considerations have been purely combinatorial and representation-theoretic. In this section we give a geometrical context for the result above. To do so we first recall the geometrical definition of standard modules \cite{GV,Nakajima1} and $q,t$-characters \cite{Nakajima2}. (An axiomatic definition of $q,t$-characters was subsequently given in \cite{Nakajima3} for simply laced types and, in all types, in \cite{Hernandezqt}.) The standard modules are defined geometrically in terms of graded quiver varieties. In type $A_1$, the relevant graded quiver variety is 
\be\mathfrak M^\bullet (U,W) \cong \{(x, E) \in \mathfrak N(W) \times \grassm U  W :  \mbox{im } x \subseteq E \subseteq \ker x\}\, ,\ee
where
\be U = \bigoplus_{a \in\Cx} U_a \, ,\qquad \qquad W = \bigoplus_{a\in\Cx} W_a\ee
are finite dimensional $\Cx$-graded vector spaces, $\grassm U W$ denotes the Grassmannian of $\Cx$-graded subspaces $E$ of $W$ such that $\dim E_a = \dim U_{aq}$ for all $a \in \Cx$ and, finally, 
\be \mathfrak N(W) := \{x \in \mbox{End}(W): \,\, x(W_a) \subseteq W_{aq^{-2}} \,\,\,\,\mbox{and} \,\,\,\, x^2=0\}\, . \ee
Letting $H_\bullet ( - , \mathbb C)$ denote the Borel-Moore homology with complex coefficients, we have
\begin{proposition}
\label{Prop:StdModulesHomology}
There exists an isomorphism of $\uqslth$-modules
\be \label{MPHomology}M(P) \cong \bigoplus_{[U]} H_\bullet (\grassm U W, \mathbb C)\, , \qquad \mbox{where} \qquad P(u) = \prod_{a \in \Cx} (1-ua)^{\dim W_a}\nn\ee 
and the direct sum runs over the isomorphism classes of $\Cx$-graded subspaces of $W$.
\end{proposition}
On the homology side, the $\uqslth$-module structure is obtained through the convolution product \cite{GV}. Proposition \ref{Prop:StdModulesHomology} is the original \emph{geometrical} definition of standard modules and it is a non-trivial result of \cite{VV} to establish that this geometrical definition agrees with the  \emph{algebraic} one given in definition \ref{Mdef}.

The $l$-weight spaces of $M(P)$ are isomorphic to the direct summands $H_\bullet (\grassm U W, \mathbb C)$ in Proposition \ref{Prop:StdModulesHomology}. With every pair $(U,W)$ of $\Cx$-graded vector spaces as above, we associate the monomial \be m_{U, W} = \prod_{a \in \Cx} Y_a^{\dim W_a - \dim U_{aq^{-1}} - \dim U_{aq}} \in \mathbb Z[Y^{\pm 1}_a]_{a \in \Cx}\, .\ee

Let $\mathscr P_{U,W} \in \mathbb Z[t]$ denote the Poincar\'e polynomial of $\grassm U W$;
\be \mathscr P_{U,W}(t) := \sum_k (-t)^k \dim H_k(\grassm U W, \mathbb C) \, .\ee
Following \cite{Nakajima3}, we make the following
\begin{definition}\label{qtdef}
The $q,t$-character of the standard module $M(P)$ with $P$ given in terms of $W$ by (\ref{MPHomology}) is the generating function of the Poincar\'e polynomials $\mathscr P_{U,W}$ of the $\Cx$-graded Grassmannians; \ie
\be \chi_{q,t}(M(P)) := \sum_{[U]} \mathscr P_{U,W}(t) \,\,m_{U,W} \, ,\ee
where the sum runs over isomorphism classe of $\Cx$-graded subspaces of $W$.
\end{definition}

We have, by construction,
\be \grassm U W \cong \prod_{a \in \Cx} \grass{\dim U_{aq}}{\dim W_a} \, ,\ee
where, for all $m \leq n \in \mathbb N$, we denote by $\grass m n$ the usual Grassmannian of $m$-dimensional subspaces of $\mathbb C^n$. 
It is enough, therefore, to focus, as we did in section 3 above, on the $n$-fold coincident case. That is, we pick an $a\in \Cx$, and set $W=W_a\cong \mathbb C^n$. Consider, in particular, the $l$-weight space given by $U=U_{aq}\cong \mathbb C^m$.
The set of Schubert classes -- \ie the set of fundamental classes of the Schubert subvarieties -- of the Grassmannian $\grass m n$, constitutes a basis of $H_\bullet (\grass m n, \mathbb C)$ \cite{Fulton}. For each $k\in \mathbb N$, the Schubert classes in $H_{2k}(\grass m n, \mathbb C)$ are in bijection with the Ferrers diagrams with $k$ boxes that fit into a rectangle of size $m \times (n-m)$. 

In this way, one arrives from the geometrical perspective  at the same combinatorial objects as in section \ref{Sec:ZA}. In particular we have the well-known result that  the Poincar\'e polynomial of $\grass m n$ is given by ${{n}\choose{m}}_{\! t^2}$, and hence that
\be \mathscr P_{U,W} (t) = \prod_{a \in \Cx} {\dim W_a \choose \dim U_{aq}}_{\!\!t^2}.\ee
Comparing definition \ref{qtdef} and theorem \ref{THX}, one sees that the dimensions of the Jordan grades of standard modules can indeed be read off from their $q,t$-characters; which proves 
\begin{theorem}
\label{TH}
Let $P \in \mathbb C[u]$ be a monic polynomial and $M(P)$ be the standard $\uqslth$-module with Drinfel'd polynomial $P$. Then, for every $l$-weight $\gamma$ of $M(P)$, there exists $n(\gamma) \in \mathbb N$ and a permutation $\sigma_\gamma \in S_{n(\gamma)}$ such that
\be \chi_{q,t}(M(P)) = \sum_{\gamma}  \left (\sum_{k=0}^{n(\gamma)-1} t^{2 \sigma_\gamma(k)} \dim(\mathscr F_{k+1} M_{\gamma}(P)\big/\mathscr F_{k}M_\gamma(P)) \right ) m_\gamma \, .\nn\ee
\end{theorem}

The geometrical construction of quantum affine algebras and of their standard modules generalizes to every quantum affine algebra of simply laced type in the context of quiver varieties \cite{Nakajima1}. In the particular case of quantum affine algebras of type $A_{N-1}$ with $N>2$ for example, the relevant quiver variety admits an explicit characterization involving $N$-step partial flags in place of Grassmannians. In that case, standard modules correspond to the Borel-Moore homology of fibers at points in the intersection of the Slodowy slice and the closure of some nilpotent orbit \cite{NakajimaALE, Maffei}. 

From that geometrical perspective, it is natural to expect that theorem \ref{TH} extends to every quantum affine algebra of simply laced type. 
To prove this using the techniques of the present paper it is necessary first to construct, by a limiting procedure as in section \ref{Sec:ZA}, bases of thick standard modules in which the  $\phi_{i}^\pm(u)$ act upper-triangularly.
In type $A_{N-1}$ this is possible at least for tensor products of the first fundamental representation, the analogs of the vectors $\ZZ A$ being labelled by $N$-step filtrations of $\range 1 n$. Further work is then needed to establish that the Jordan blocks are of the expected dimensions.

\vspace{1cm}

\noindent\textsl{Funding}\,\, The work of C.A.S.Y. was supported by a fellowship from the Japan Society for the Promotion of Science, ID number P09771. 

\small
\bibliography{qtsl2}    

\providecommand{\bysame}{\leavevmode\hbox to3em{\hrulefill}\thinspace}
\providecommand{\MR}{\relax\ifhmode\unskip\space\fi MR }
\providecommand{\MRhref}[2]{%
  \href{http://www.ams.org/mathscinet-getitem?mr=#1}{#2}
}
\providecommand{\href}[2]{#2}
\begin{thebibliography}{Nak01b}

\bibitem[CP]{CPbook}
V.~Chari and A.~Pressley, \emph{{A guide to quantum groups}}, Cambridge, UK:
  Univ. Pr. (1994) 651 p.

\bibitem[CP91]{CPsl2}
\bysame, \emph{Quantum affine algebras}, Comm. Math. Phys. \textbf{142} (1991),
  261--283.

\bibitem[CP01]{CPweyl}
\bysame, \emph{{Weyl modules for classical and quantum affine algebras}},
  Representation Theory \textbf{5} ({2001}), 191--223.

\bibitem[Dri87]{Drinfeld1}
V.~G. Drinfeld, \emph{{Quantum groups}}, Proc. Int. Cong. Math. (Berkeley,
  1986) \textbf{1} (1987), 798--820.

\bibitem[Dri88]{Drinfeld2}
\bysame, \emph{{A New realization of Yangians and quantized affine algebras}},
  Sov. Math. Dokl. \textbf{36} (1988), 212--216.

\bibitem[FM01]{FM}
E.~Frenkel and E.~Mukhin, \emph{{Combinatorics of q-characters of
  finite-dimensional representations of quantum affine algebras}}, Commun.
  Math. Phys. \textbf{216} (2001), 23--57.

\bibitem[FR98]{FR}
E.~Frenkel and N.~Reshetikhin, \emph{{The q-characters of representations of
  quantum affine algebras and deformations of W-algebras}}, Contemp. Math.
  \textbf{248} ({1998}), 163--205.

\bibitem[Ful]{Fulton}
W.~Fulton, \emph{{Young tableaux}}, Cambridge, UK: Univ. Pr. (1997) 260 p.

\bibitem[Gro07]{Grosse}
P.~Gross\'e, \emph{{On quantum shuffle and quantum affine algebras}}, J.
  Algebra (2007), no.~2, 495--519.

\bibitem[GV93]{GV}
V.~Ginzburg and E.~Vasserot, \emph{{Langlands reciprocity for affine quantum
  groups of type $A_n$}}, International Mathematics Research Notices
  \textbf{1993} ({1993}), no.~3, 67--85.

\bibitem[Her04]{Hernandezqt}
D.~Hernandez, \emph{{Algebraic Approach to $q,t$-Characters}}, Advances in
  Mathematics \textbf{187} (2004), no.~1, 1--52.

\bibitem[Her05]{HernandezFusionI}
\bysame, \emph{{Representations of Quantum Affinizations and Fusion Product}},
  Transformation Groups \textbf{10} (2005), 163--200.

\bibitem[Her07a]{HernandezFusionII}
\bysame, \emph{{Drinfeld coproduct, quantum fusion tensor category and
  applications}}, Proc. London Math. Soc. \textbf{95} (2007), no.~3, 567--608.

\bibitem[Her07b]{Hminaffs}
\bysame, \emph{{On minimal affinizations of representations of quantum
  groups}}, Comm. Math. Phys. \textbf{276} (2007), no.~1, 221--259.

\bibitem[Jim85]{Jimbo}
M.~Jimbo, \emph{{A q difference analog of U(g) and the Yang-Baxter equation}},
  Lett. Math. Phys. \textbf{10} (1985), 63--69.

\bibitem[Kni95]{Knight}
H.~Knight, \emph{{Spectra of Tensor Products of Finite Dimensional
  Representations of Yangians}}, Journal of Algebra \textbf{174} (1995), no.~1,
  187 -- 196.

\bibitem[Mac]{MacDonald}
I.~G. MacDonald, \emph{{Symmetric functions and Hall polynomials}}, Oxford
  Univ. Pr. (1998) 488 p.

\bibitem[Maf05]{Maffei}
A.~Maffei, \emph{{Quiver varieties of type $A$}}, Comment. Math. Helv.
  \textbf{80} (2005), 1--27.

\bibitem[Nak94]{NakajimaALE}
H.~Nakajima, \emph{{Instantons on ALE spaces, quiver varieties, and Kac-Moody
  algebras}}, Duke Math. J. \textbf{76} (1994), 365--416.

\bibitem[Nak01a]{Nakajima1}
\bysame, \emph{{Quiver varieties and finite dimensional representations of
  quantum affine algebras }}, J. Amer. Math. Soc. \textbf{14} ({2001}),
  145--238.

\bibitem[Nak01b]{Nakajima2}
\bysame, \emph{{$t$-analogs of the $q$-characters of finite dimensional
  representations of quantum affine algebras}}, Physics and Combinatorics,
  Proceedings of the Nagoya 2000 International Workshop (2001), 195--218.

\bibitem[Nak04]{Nakajima3}
\bysame, \emph{{Quiver varieties and $t$-analogs of $q$-characters of quantum
  affine algebras}}, Annals of mathematics \textbf{160} ({2004}), no.~3,
  1057--1097.

\bibitem[Pro82]{Proctor}
R.~A. Proctor, \emph{{Solutions of two difficult combinatorial problems with
  linear algebra}}, The American Mathematical Monthly \textbf{89} ({1982}),
  no.~10, 721--734.

\bibitem[Tho00]{Thoren2000iro}
J.~Thor{\'e}n, \emph{{Irreducible Representations of Quantum Affine Algebras}},
  ProQuest LLC, 2000.

\bibitem[Vas98]{Vass}
E.~Vasserot, \emph{{Affine quantum groups and equivariant K-theory}},
  Transformation Groups \textbf{3} ({1998}), no.~3, 269--299.

\bibitem[VV02]{VV}
M.~Varagnolo and E.~Vasserot, \emph{{Standard modules of quantum affine
  algebras}}, Duke Math. J. \textbf{111} ({2002}), no.~3, 509--533.

\bibitem[VV03]{VVPervSheav}
\bysame, \emph{{Perverse sheaves and quantum Grothendieck rings}}, Studies in
  memory of Issai Schur, Progr. Math. \textbf{210} ({2003}), 345--365,
  Birkhauser Boston.

\end{thebibliography}

\bibliographystyle{amsalpha}    

\end{document}